\documentclass[12pt]
{amsart}

\usepackage[utf8x]{inputenc}
\usepackage{amsmath,amsthm,amsfonts,amssymb,mathdots,amscd}
\usepackage{float}
\usepackage{epsfig, psfrag}
\usepackage{graphicx,graphics}
\usepackage{hyperref} 
\usepackage{latexsym}
\usepackage[all]{xy}
\usepackage[usenames,dvips]{color}  
\usepackage{subfiles}
\usepackage{caption}
\usepackage{subcaption}
\usepackage{lscape}
\usepackage{tabularx} 
\usepackage{enumerate}

\newcolumntype{b}{X}
\newcolumntype{s}{>{\hsize=.5\hsize}X}

\xyoption{arrow}

\theoremstyle{plain}
\newtheorem{lemma}{Lemma}[section]
\newtheorem{prop}[lemma]{Proposition}
\newtheorem{cor}[lemma]{Corollary}
\newtheorem{thm}[lemma]{Theorem}

\newtheorem{de}[lemma]{Definition}

\newtheorem{rem}[lemma]{Remark}

\newtheorem{problem}[lemma]{Problem}

\newtheorem*{thmn}{Theorem}

\newcommand{\Z}{\mathbb Z}

\newcommand{\N}{\mathbb N}

\begin{document}

\title[]{Combinatorial model for the cluster categories of type E}

\author[L. Lamberti]{Lisa Lamberti}
\address{Mathematical Institute \\
University of Oxford \\
Oxford\\
OX2 6GG\\
United Kingdom 
}
\email{Lisa.Lamberti@maths.ox.ac.uk}
\maketitle

\begin{abstract}
In this paper we give a geometric-combinatorial description 
of the cluster categories of type $E$. In particular, we
give an explicit geometric description of all cluster tilting objects in 
the cluster category of type $E_6$. 
The model we propose here arises from combining two polygons,
and it generalises the description of
the cluster category of type $A$ and $D$.
\end{abstract}

\section{Introduction}
Caldero-Chapoton-Schiffler defined in \cite{CCS} 
categories arising from homotopy classes of paths 
between two vertices of a regular $(n+3)$-sided polygon. 
Independently, Buan-Marsh-Reiten-Reineke-Todorov 
defined cluster categories 
as certain orbit categories of the 
bounded derived category of hereditary algebras, see \cite{BMRRT}. 
The latter are 
algebras arising from oriented graphs $Q$ with no oriented cycles. 
When $Q$ is an orientation of 
a Dynkin graph of type $A_n$, the category
constructed in \cite{CCS} coincides with the one of \cite{BMRRT}.
In this paper we model a number of different 
orbit categories arising from 
orientations of tree graphs $T_{r,s,t}$ in geometric terms. 
The description we
propose here is based on the idea of doubling 
the set of oriented diagonals in a given regular polygon 
and combine the dynamics of these two sets in an appropriate way.

Among the additive categories we can model in this way we find cluster categories 
of type $T_{r,s,t}$, as well as other triangulated categories of various Calabi-Yau dimension.
All these categories arise as mesh categories 
of a translation quiver whose vertices are
single coloured oriented and paired diagonals 
in a polygon $\Pi$. 
We will see that polygons of different sizes yield 
different categories and these are Calabi-Yau of dimension 
two only in few exceptional cases. The advantage of our approach is that we can
model the combinatorics of these various categories in terms
of configurations of diagonals in $\Pi$. 

For a tree diagram $T_{r,s,t}$ let $\mathcal{C}_{T_{r,s,t}}$ be
cluster category of type $T_{r,s,t}$ defined as the orbit
category $\mathcal{D}^b(\mathrm{ mod } k T_{r,s,t})/\tau^{-1}\Sigma$.
Then for $n\geq\max\{r+t+1,r+s+1\}$ and a regular $(n+3)$-gon $\Pi$, 
we construct  an additive category $\mathcal{C}^{n+3}_{r,s,t}$ of coloured
oriented single and paired diagonals of $\Pi$. 

With this notation our first main result can be stated as follows.
\begin{thmn}[\ref{eqofcat}]
We have the following equivalences of additive categories:
\begin{align*} \itemsep0em
&\mathcal{C}^7_{1,2,2}\rightarrow \mathcal{C}_{E_6} \\
&\mathcal{C}^{10}_{1,2,3}\rightarrow \mathcal{C}_{E_7}\\
&\mathcal{C}^{16}_{1,2,4}\rightarrow \mathcal{C}_{E_8}.
\end{align*}
\end{thmn}
When $T_{r,s,t}$ is not of Dynkin type we obtain an equivalence from 
the category $\mathcal{C}^\infty_{r,s,t}$ associated to the 
infinite sided-polygon $\Pi^\infty$ 
to the full subcategory of the cluster category $\mathcal{C}_{T_{r,s,t}}$ 
with indecomposable objects in the transjective component of the AR-quiver
of $\mathcal{C}_{T_{r,s,t}}$. 

These equivalences enable us to investigate
the combinatorics of the cluster category
in geometric terms. 
More precisely, we are able to
describe all 833 cluster tilting sets of the cluster category of type $E_6,$
$\mathcal{C}_{E_6}$, as cluster configuration of six single coloured 
oriented and paired diagonals
in a heptagon. The strategy will be to 
first determine two fundamental families of cluster configurations
$\mathcal{F}_1$ and $\mathcal{F}_2$ and deduce the
remaining cluster tilting sets
using the rotation inside $\Pi$, induced from the Auslander-Reiten 
translation $\tau$ in $\mathcal{C}_{E_6}$,
as well as a symmetry $\sigma$ of our model.
\begin{thmn}[\ref{count}] 
In a heptagon
\begin{itemize}
 \item 350 different cluster configurations have one long paired diagonal, and 
they arise from $\mathcal{F}_1$ through $\tau$.
 \item 483 other cluster configurations arise from  
 $\mathcal{F}_2$ through $\sigma$ and $\tau$.
\end{itemize}
\end{thmn}
This classification allows us to deduce the following result.
\begin{thmn}[\ref{confthm}]
All 833 cluster tilting sets of $\mathcal{C}_{E_6}$ can be expressed as
configurations of six non-crossing coloured oriented single and paired diagonals 
inside two heptagons.
\end{thmn}
The previous results enable us to
deduce geometrical moves describing the mutation process between 
cluster tilting objects in $\mathcal{C}_{E_6}$.
The geometrical moves we find extend 
the mutation process of cluster categories of type $A_n$,
as described by Caldero-Chapoton-Schiffler in \cite{CCS}, to the setting of 
coloured oriented diagonals. 
We will see that in many cases
mutations inside $\mathcal{C}_{E_6}$ correspond 
to flips of coloured oriented diagonals in $\Pi$. 
In the final part of the paper we use the previous results to 
categorify geometrically cluster algebras of type $F_4$. We also provide a 
geometric description of the mutation rule.

\textbf{Relation to previous work:}
Geometrical models for cluster categories of other types have been investigated also in
 \cite{S,BZ,to}.  Moreover, with an appropriate paring of coloured oriented single and paired diagonals in an even sided polygon 
we recover the categorified geometric realisation of cluster algebras of type $D$ given by Fomin-Zelevinsky in \cite{FZ_y}.

The number of clusters in a cluster algebra of finite type
was first computed in \cite{FZ_y}. Under the bijection
of \cite{BMRRT} one deduces the number of cluster tilting sets
in the corresponding cluster categories. 
An explicit complete description of the
cluster-tilting objects in the cluster category of type $E_6$ and $F_4$ however is new.

In addition,  Fomin-Pylyavskyy used in \cite{FP}
polygons to describe the cluster algebra structure in certain rings of 
$\mathrm{SL}(V)$-invariants. More precisely, they 
construct invariants determined by tensor graphs
associated to diagonals of polygons. Using
a heptagon, Fomin-Pylyavskyy also model the cluster algebra structure
of the homogeneous coordinate ring $\mathbb C[Gr_{3,7}]$
of the affine cone over the Grassmannian $Gr_{3,7}$ of three 
dimensional subspaces in a seven dimensional 
complex vector space. By a result of Scott, \cite{Scott}, 
it is known that the ring $\mathbb C[Gr_{3,7}]$ 
is a cluster algebra type of $E_6$. 
The approach of \cite{FP} however is different 
then the one we propose here,
as it relies on relations satisfied by tensor graphs,
called skein relations of tensor graphs.

\textbf{Organisation of the article:}
In Section 2, we state some preliminary results and definitions. In particular,
we state the fundamental properties of orbit categories and we remind 
the action of the shift functor on the Auslander-Reiten quiver
of $\mathcal{D}^b(\mathrm{mod}k Q)$, for $Q$ an orientation of a simply laced
Dynkin diagram.

In Section 3, we construct the additive category $\mathcal{C}^{n+3}_{r,s,t}$  
associated to a regular (n+3)-gon $\Pi$, where $n=\max\{r+t+1,r+s+1\}$. 
The objects of $\mathcal{C}^{n+3}_{r,s,t}$ will be 
single coloured oriented and paired diagonals of $\Pi$. The morphism spaces 
are generated by minimal rotations in $\Pi$, modulo 
certain equivalence relations. 

In Section 4 we prove the equivalences of additive categories 
stated in Theorem \ref{eqofcat}.
Further equivalence of additive categories will also be discussed. 
In Proposition \ref{isoinfty} we show that there
is an equivalence between the category $\mathcal{C}^\infty_{r,s,t}$ 
associated to an infinite sided polygon and the full subcategory of the 
cluster category $\mathcal{C}_{T_{r,s,t}}$, whose 
indecomposable objects belong to 
the transjective component of the AR-quiver
of $\mathcal{C}_{T_{r,s,t}}$.

In Section 5, we describe the combinatorics of $\mathcal{C}_{E_6}$ geometrically 
inside a heptagon $\Pi$. 
In Theorem \ref{count} and Theorem 
\ref{confthm} we describe all cluster tilting sets of $\mathcal{C}_{E_6}$ in terms of 
cluster configurations of coloured oriented
diagonals of $\Pi$. In Proposition \ref{int} we also describe all $\mathrm{Ext}$-spaces 
of $\mathcal{C}_{E_6}$ using curves 
of coloured oriented diagonals of $\Pi$. Finally, results concerning the mutation
process of cluster configurations will be stated, see Proposition
\ref{378}. 

In Section 6 we point out further applications of our work. In particular
we deduce a geometric additive categorification of cluster algebras of type $F_4$.
Moreover, we describe how our construction can
be used to understand cluster tilting sets inside the cluster
categories of type $E_7$ and $E_8$, as well as 
cluster tilting sets in the transjective component of the 
AR-quiver of cluster categories associated to
more general tree diagrams $T_{r,s,t}$.

\section{Preliminaries}
Let $k$ be an algebraically closed field and let $Q$ be an acyclic quiver.
Let $\mathrm{mod}kQ$ be the abelian category
of $k$-finite dimensional right-modules over
the path algebra $kQ$. 
Let $\mathcal{D}:=\mathcal{D}_Q:=\mathcal{D}^b(\mathrm{mod} k Q)$ be the bounded derived category of
$\mathrm{mod}kQ$ endowed with the shift functor
$\Sigma:\mathcal{D}\rightarrow\mathcal{D}$ and the
{\em Auslander-Reiten translation} $\tau:\mathcal{D}\rightarrow \mathcal{D}$
characterised by $\mathrm{Hom}_{\mathcal{D}}(X,-)^*\cong\mathrm{Hom}_{\mathcal{D}}(-,\Sigma\tau X),$ for all $X\in\mathcal{D}$.

\subsection{Orbit categories of $\mathcal{D}$}\label{def_orbitcat}
We are interested in the orbit categories $\mathcal{C}_Q^p$
of $\mathcal{D}$, $p\in\mathbb N$, generated by 
the action of cyclic group generated by the auto-equivalences 
$F^p:=(\tau^{-1}\Sigma)^p=\tau^{-p}\Sigma^p$. 
The objects of $\mathcal{C}_Q^p$ are the same as the objects of $\mathcal{D}$ and 
$$
\mathrm{Hom}_{\mathcal{C}_Q^p}(X,Y):=\bigoplus_{t\in\mathbb Z}
\mathrm{Hom}_{\mathcal{D}}(X,(F^p)^tY).
$$
Morphisms are composed in a natural way.

When $p=1$, $\mathcal{C}_Q:=\mathcal{C}^1_Q$ is 
the {\em cluster category of type $Q$}
defined in \cite{BMRRT}, and independently
in \cite{CCS} in geometric terms for $Q$ of type $A_n$. In all other cases
$\mathcal{C}^p_Q$ is 
the {\em $p$-repetitive cluster category} studied by the 
author in \cite{Lisa2} for $Q$ of type $A_n$, 
and introduced by Zhu in
\cite{Zhu} for $Q$ an acyclic quiver. 

\subsection{Fundamental properties of orbit categories} 
Like $\mathcal{D}$, the categories $\mathcal{C}^p_Q$
are Krull-Schmidt 
and have finite dimensional $\mathrm{Hom}$-spaces.  
The categories $\mathcal{C}^p_Q$ are triangulated categories, 
and the projection functor
$\pi_i:\mathcal{D}\rightarrow \mathcal{C}^p_Q$, $i\in\mathbb N$ 
is a triangle functor,
see \cite[Theorem 1]{Keller}. The induced
shift functor is again denoted by $\Sigma$.
Moreover, the categories $\mathcal{C}^p_Q$ have AR-triangles and the AR-translation 
$\tau$ is induced from $\mathcal{D}$.
The categories $\mathcal{C}^p_Q$ also have the Calabi-Yau property, i.e.
$(\tau\Sigma)^m\stackrel{\sim}{\longrightarrow}\Sigma^n $
as triangle functors, here we identify $\tau\Sigma$ with 
the Serre functor of $\mathcal{C}^p_Q$.
In particular, in $\mathcal{C}_Q$ we have that $n=2$ and $m=1$,
hence $\mathcal{C}_Q$ is Calabi-Yau of dimension $2$. In
$\mathcal{C}^p_Q$ we have that $m=2$ and $n=p$ in the above
isomorphism of triangle functors, thus $\mathcal{C}^p_Q$  is said to be a 
Calabi-Yau category of fractional dimension $\frac{p}{2}$. Notice that a 
Calabi-Yau category of fractional dimension
$\frac{p}{2}$ is in general not a Calabi-Yau category of dimension 2. 
In this paper $p\in\{1,2\}$, moreover we adopt the convention 
$\mathrm{Ext}^i_{{\mathcal{C}^p}_Q}(X,Y):=\mathrm{Hom}_{{\mathcal{C}^p}_Q}(X,\Sigma^i Y)$.

\subsection{Auslander-Reiten quiver of a Krull-Schmidt category}

A {\em stable translation quiver} $(\Gamma,\tau)$ in the sense of Riedtmann, \cite{Riedtmann},
is a quiver $\Gamma$ without loops nor multiple edges, together with a 
bijective map $\tau:\Gamma\rightarrow\Gamma$ 
called {\em translation} such that for all vertices $x$ in $\Gamma$ 
the set of starting points of arrows which end 
in $x$ is equal to the set of end points of arrows which start at $\tau(x)$.

For $(\Gamma,\tau)$ one defines the mesh category as the quotient 
category of the additive path category of $\Gamma$ by the mesh ideal, see for example  \cite{K2}. 
In particular, the mesh category of $(\Gamma,\tau)$ is an additive category.

In the next result, let $\Z Q$ be the repetitive quiver of $Q$, see \cite[I,5.6]{Happel2} for a reminder 
on this construction. 
Let $\tau: \mathbb Z Q\rightarrow\mathbb Z Q$ be the automorphism
defined on the vertices $(n,i)$ of $\mathbb Z Q$ by 
$\tau(n,i)=(n-1,i)$, 
for $n\in\mathbb Z$, $i$ a vertex of $Q$.

\begin{thm}\cite[I.5.5]{Happel2}\label{thm presentation derived cat}
Let $kQ$ be a finite dimensional hereditary $k$-algebra. If $Q$ is an orientation of
\begin{itemize}\itemsep0em
\item a simply laced Dynkin graph,
$AR(\mathcal{D}^b(\mathrm{mod}\,kQ))$ is isomorphic (as stable translation quiver) 
to $\mathbb Z Q$.
\item an affine graph, 
$AR(\mathcal{D}^b(\mathrm{mod}\,kQ))$ splits into
components of the form $  \mathbb Z Q$ and $\mathbb Z A_\infty/r$, for some $r\in \mathbb N$.
\item a wild graph, the components of 
$AR(\mathcal{D}^b(\mathrm{mod}\,kQ))$ are of 
the form $ \mathbb Z Q$ 
and $\mathbb Z A_\infty$.
\end{itemize}
\end{thm}

Let $\mathrm{ind}\, \mathcal{D}$ be the full 
subcategory of $\mathcal{D}$
of indecomposable objects.
\begin{thm}\cite[I,5.6]{Happel2}\label{Happel2}
Let $Q$ be an orientation of a simply laced Dynkin graph. 
The mesh category of $(\mathbb{Z} Q,\tau)$ 
is equivalent to $\mathrm{ind }\,\mathcal{D}$. 
\end{thm}
A first important consequence of this result is that the AR-quiver 
of $\mathcal{D}$ is independent of the orientation of $Q$. 

\subsection{Induced action of $\Sigma$ on $\Z Q$ }\label{action}
In this section let $Q$ be an orientation of a simply-laced Dynkin graph.  
Below we point out some known facts about the induced action of $\Sigma$ and $\tau$ 
on $\Z Q$ taken from \cite[Chap. 4]{Miyachi}, see also \cite[Chap. 4]{Jorgensen}.
These considerations, together with Theorem \ref{thm presentation derived cat}, 
will enable us to determine the precise shape of the AR-quiver 
of various orbit categories investigated in the sequel.  

The induced action of $\tau$ on $\mathbb Z Q$ is always an 
horizontal shift to the left.
The induced action of $\Sigma$ on $\mathbb Z A_n$ coincides with
a shift of $\frac{n+1}{2}$ units to the right, composed with a reflection 
along the horizontal central line of $\mathbb Z A_n$.
On $\mathbb Z D_n$ the action of $\Sigma$
agrees with $\tau^{-(n-1)}$ composed with
an order two automorphism $\rho$ defined 
on $\mathbb{Z} D_n$ when $n$ is odd. 
While on $\mathbb{Z} E_6$ the action of $\Sigma$
coincides with the action of $\rho\tau^{-6}$ where
$\rho$ is an automorphism of order two defined on $\mathbb{Z} E_6$.
Moreover, $\Sigma$ acts as $\tau^{-9}$ on $\mathbb{Z} E_7$, 
and as $\tau^{-15}$ on $\mathbb{Z} E_8$. 

\subsection{The repetitive cluster category $\mathcal{C}^p_{A_{n}}$}

The geometrical model for the cluster 
categories of type $E$
we propose in this paper is motivated from the following idea: glue together
two copies of AR-quiver of $\mathcal{C}^2_{A_{n}}$. 
Let us remind the reader some facts
about this category.

The previous discussion implies that
$(AR(\mathcal{C}^p_{A_{n}}),\tau)\cong(\mathbb Z A_{n}/(\tau^{-1}\Sigma)^p,\tau)$, for $p\in\N$.
Moreover, $(AR(\mathcal{C}^p_{A_{n}}),\tau)$ 
can be modelled using diagonals in polygons
as done in \cite{Lisa2}.
When $p=2$, the quiver
$(AR(\mathcal{C}^2_{A_{n}}),\tau)$ can entirely 
be modelled using oriented diagonals in a regular $(n+3)$-gon.
In Figure \ref{ARA4} an illustration of 
this construction is provided for $p=2,n=4$. 
\begin{figure}[h]
\center
\includegraphics[scale=0.5]{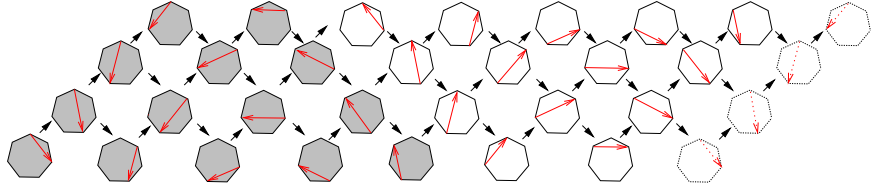}
\caption{AR-quiver of $\mathcal{C}^2_{A_4}$.}
\label{ARA4}
\end{figure}

\section{Single coloured oriented and paired diagonals in polygons}\label{sect3}

Throughout the rest of the paper let $T_{r,s,t}$ be 
an orientation of a finite connected 
graph with $r+s+t+1$ vertices and three legs. We assume 
the legs of $T_{r,s,t}$ to have $r$, resp.
$s$, resp. $t$ vertices and that one vertex of $T_{r,s,t}$ has three neighbours. 
We say that a tree $T_{r,s,t}$ is {\em symmetric} if $s=t$. 

Unless specified otherwise, for $n\geq\mathrm{max}\{r+s+1;r+t+1\}$ let 
$\Pi$ be a regular $(n+3)$-gon with vertices numbered in the 
clockwise order by the group $\mathbb Z/(n+3)\mathbb Z$. 

For vertices $i,j,k$ of $\Pi$ we write $i\leq j\leq k$
if $j$ is between $i$ and $k$ in the clockwise order.
Moreover, we denote by $(i,j)$ the 
unoriented diagonal of $\Pi$
joining the vertices $i$ and $j$ and by 
$[i,j]$ the oriented diagonal 
of $\Pi$ starting at $i$ and ending in $j$. 
We do not consider boundary segments 
as oriented diagonals.

\subsection{Single coloured oriented and paired diagonals of $\Pi$}\label{pairs}

We start describing the geometric 
construction leading to the
modelling of a number of orbit categories of 
$\mathcal{D}^b(\mathrm{mod}\, k T_{r,s,t})$ 
arising from orientations of symmetric trees $T_{r,s,t}$.

To begin the construction we double
the set of oriented diagonals of $\Pi$,
and distinguish each set with colours 
using subscripts $R$, $B$
e.g. $[1,3]_R$ is the red diagonal  
linking the vertex 1 to 3 of $\Pi$. 
For every vertex $i$ of $\Pi$ we form 
the following $(r+1)$ pairs of 
coloured oriented diagonals:
\begin{align*}\itemsep0em
[i,i+2]_P=&[i+2,i]_P:=\{ [i,i+2]_R,[i+2,i]_B\}\\
[i,i+3]_P=&[i+3,i]_P:=\{ [i,i+3]_R,[i+3,i]_B\}\\
&\dots\\
[i,i+r+2]_P=&[i+r+2,i]_P:=\{ [i,i+r+2]_R,[i+r+2,i]_B\}.
\end{align*}
When $r=0$, we assume that there are no 
paired diagonals. Moreover, let us point out that
$[i,j]_P\neq[j,i]_P$.

Next we define a subset of
$(r+1)(n+3)$ paired, $s(n+3)$ single red and $t(n+3)$ single blue oriented 
diagonals of $\Pi$ as follows:
\begin{align*}\itemsep0em
\Pi_{r,s,t}:=\bigg\{&[i,i+2]_P, \dots,[i,i+r+2]_P, \\
&[i,i+r+3]_R,\dots,[i,i+r+s+2]_R, \\
&[i+r+3,i]_B, \dots,[i+r+t+2,i]_B,\hspace{0,3cm} \textrm{ $i$ vertex  of }\Pi\bigg\}.
\end{align*}
Once coloured oriented 
diagonals are paired, they stop existing as single 
coloured oriented diagonals in $\Pi_{r,s,t}$.

Consider the subset of elements of $\Pi_{r,s,t}$ given by
$\Pi_{r,s,t}\vert_1:=\{[1,3]_P, \dots,[1,r+3]_P, 
[1,r+4]_R,\dots,[1,r+s+3]_R, 
[r+4,1]_B, \dots,[r+t+3,1]_B\}$. Then  on the one side 
elements of $\Pi_{r,s,t}\vert_1$ are in bijection with
the vertices of $T_{r,s,t}$. On the other side,  elements of $\Pi_{r,s,t}\vert_1$ give rise to a 
triangulation of a region
inside $\Pi$ homotopic to 
a regular $(r+s+4)$-gon, resp. to a $(r+t+4)$-gon. 

In Figure \ref{projdiag} an illustration of this situation is provided. 
Coloured oriented diagonals with the same label are identified.
The black vertices in the figure represent the vertices of $T_{r,s,t}$.
The dotted lines are the edges of $T_{r,s,t}$. 
Vertices on, and edges between, identified paired diagonals 
give rise to one vertex, and one edge, in $T_{r,s,t}$.

\begin{figure}[H]
\psfragscanon
 \psfrag{i}[][][0.65]{$i$}
  \psfrag{i+1}[][][0.65]{$i+1$}
  \psfrag{i+2}[][][0.65]{$i+2$}
  \psfrag{i+r+1}[][][0.65]{$i+r+2$}
  \psfrag{i+r+2}[][][0.65]{$i+r+3$}
  \psfrag{i+r+3}[][][0.65]{$i+r+4$}
  \psfrag{i+r+s+2}[][][0.65]{$i+r+s+2$}
   \psfrag{i+r+s+3}[][][0.65]{$i+r+s+3$}
   \psfrag{i+r+t+2}[][][0.65]{$i+r+t+2$}
   \psfrag{i+r+t+3}[][][0.65]{$i+r+t+3$}
    \psfrag{d}[][][0.5]{$\dots$}
\includegraphics[scale=0.4]{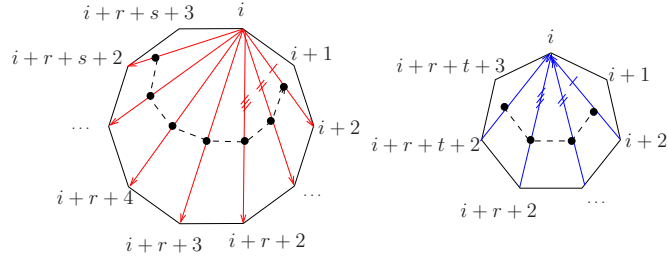}
\caption{The coloured oriented diagonals in bijection 
with the vertices of $T_{2,4,1}$.}
\label{projdiag}
\end{figure}

\subsection{The automorphisms $\rho$ and $\tau$}\label{rhotau}
Our next aim is to define two automorphisms: $\rho$
and $\tau$, acting on the set of coloured oriented single an paired diagonals
of $\Pi$ associated to a tree 
$T_{r,s,t}$. The first automorphism is induced from the graph automorphism 
of a symmetric tree, hence only defined on $\Pi_{r,t,t}.$ The definition of
the second automorphism depends on the parity of the number of sides
of $\Pi$.

Let $c\in\{R,B,P\}$. Then we define $\rho:\Pi_{r,t,t}\rightarrow \Pi_{r,t,t}$ 
as the automorphism of order two
given by  
$$
\rho\big([i,j]_c\big):=\begin{cases}
[j,i]_B & \textrm{if } c=R, \\
[j,i]_R  &\textrm{if } c=B,\\
[i,j]_P &\textrm{otherwise.}
\end{cases}
$$

\begin{figure}[H]
\psfragscanon
 \psfrag{1}[][][0.5]{$1$}
  \psfrag{2}[][][0.5]{$2$}
   \psfrag{3}[][][0.5]{$3$}
    \psfrag{4}[][][0.5]{$4$}
     \psfrag{5}[][][0.5]{$5$}
      \psfrag{6}[][][0.5]{$6$}
       \psfrag{7}[][][0.5]{$7$}
             \psfrag{r}[][][0.5]{$\rho$}
\includegraphics[scale=0.28]{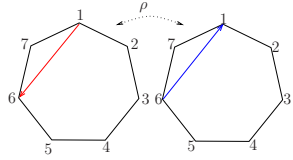}
\caption{The symmetry $\rho$.}
\label{rho}
\end{figure}
Moreover, we define the automorphism $\tau:\Pi_{r,s,t}\rightarrow \Pi_{r,s,t}$ as follows:
if $s\neq t$, then 
$
\tau([i,j]_c):=[i-1,j-1]_c.
$ If $s=t$, then
$$
\tau([i,j]_c):=\begin{cases}
\rho^{(n+3)}\big([i-1,j-1]_c\big)  &\textrm{ if }[i,j]_c\in  \Pi_{r,t,t}\vert_1 \\
[i-1,j-1]_c & \textrm{ otherwise.}
\end{cases}
$$

Geometrically, the action of $\tau$  is given by the 
anticlockwise rotation through $\frac{2\pi}{n+3}$
around the centre of $\Pi$, on all elements different then the diagonals in $\Pi_{r,t,t}\vert_1$
associated to a symmetric tree in a polygon with an odd number of sides. On the
latter the rotation is followed by the simultaneous change of colour and orientation.

\subsection{Minimal clockwise rotations}\label{rotations}\label{quivegamma}
Minimal clockwise rotations for unoriented diagonals have
been introduced in \cite[\S 2]{CCS} with the aim 
of modelling irreducible morphisms
in the cluster category $\mathcal{C}_{A}$. 
Following the spirit of \cite{CCS} we now define
minimal rotations between diagonals of $\Pi_{r,s,t}$. 

Let $k,l$ be non-neighbouring vertices of $\Pi$ and 
let $c\in\{R, B, P\}$. Let $\Pi_{r,s,t}$ be the set of coloured oriented
single and paired diagonals in $\Pi$ associated to an asymmetric tree $T_{r,s,t}$. 
Then the following three operations between diagonals in 
$\Pi_{r,s,t}$ are called {\em  minimal clockwise rotation}:
$[k,l]_c\rightarrow [k,l+1]_c$ and $[k,l]_c\rightarrow[k+1,l]_c$  
and
$$
\xymatrix@R=0.5pc @C=0.5pc{
&&[k,k+r+3]_R\ar[rd]&&\\
&[k,k+r+2]_P \ar[ru]\ar[rd]& &[k+1,k+r+3]_P. \\
&&[k+r+3,k]_B\ar[ur]&&\\
}
$$
Next, let $\Pi_{r,t,t}$ be the set of coloured oriented
single and paired diagonals in $\Pi$ associated to a symmetric tree $T_{r,t,t}$. 
Then minimal clockwise rotations are defined as before
with the following adjustment: when $[k,l]_c $,
$[k,k+r+2]_P$ are in $\tau(\Pi_{r,t,t}\vert_1)$ and $\Pi$ is odd sided, then 
we also change simultaneously the colour an the orientation. 

In Figure \ref{D4} and Figure \ref{ARE6} illustrations of the 
three minimal rotations described above
can be found.

The next remark can be used to model orbit categories arising from
orientations of tree graphs  $T_{r_1,r_2,\dots,r_m}$ where 
one vertex has $m$ neighbours.
\begin{rem}
There are three minimal rotations linking the three types of coloured oriented
single and paired diagonals of $\Pi_{r,r,t}$, namely the red, 
the blue and the paired ones.
Starting with oriented diagonals coloured in $m-1$ 
ways, one can describe the $m$ minimal rotations linking the $m$ types 
of coloured oriented diagonals with a similar diagram as above.
\end{rem}

\subsection{Quivers of single coloured oriented 
and paired diagonals of $\Pi_{r,s,t}$}\label{stabtrans}

Let $\Gamma^{n+3}_{r,s,t}$ be the quiver whose vertices are the 
elements of $\Pi_{r,s,t}$. An arrow between two vertices of $\Gamma^{n+3}_{r,s,t}$ 
is drawn whenever there is a minimal clockwise rotation linking them. No arrow is drawn otherwise. 

Concerning the shape of $\Gamma^{n+3}_{r,s,t}$ we remark 
that $\Gamma^{n+3}_{r,s,t}$ always lies on a cylinder, except when $\Pi$ is odd-sided, 
and $s=t$. Then we say that
$\Gamma^{n+3}_{r,t,t}$ lies on a 
M\"obius strip, since
the $\tau$-orbits of single
coloured oriented diagonals are twice as long as 
the $\tau$-orbits of paired diagonals.

In Figure \ref{D4} the quiver $(\Gamma^{6}_{1,1,1},\tau)$ 
of coloured oriented diagonals in
a hexagon, as well as the quiver
$(\Gamma^7_{1,2,2},\tau)$ associated to 
a heptagon are illustrated. In each quiver 
we indicate in the last slice
the identifications occurring. On both quivers 
the action of $\tau$ is always 
given by a shift to the left.

\begin{figure}[H]
\psfragscanon
 \psfrag{1}[][][0.5]{$1$}
  \psfrag{2}[][][0.5]{$2$}
   \psfrag{3}[][][0.5]{$3$}
    \psfrag{4}[][][0.5]{$4$}
     \psfrag{5}[][][0.5]{$5$}
      \psfrag{6}[][][0.5]{$6$}
\includegraphics[scale=0.5]{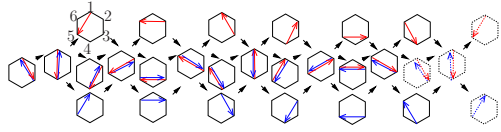}
\caption{The quiver $(\Gamma^6_{1,1,1},\tau)$ corresponds to the AR-quiver of
the orbit category $\mathcal{D}^b(\mathrm{mod}k D_4)/\Sigma^2$. }
\label{D4}
\end{figure}

\begin{figure}[H]
\psfragscanon
 \psfrag{1}[][][0.5]{$1$}
  \psfrag{2}[][][0.5]{$2$}
   \psfrag{3}[][][0.5]{$3$}
    \psfrag{4}[][][0.5]{$4$}
     \psfrag{5}[][][0.5]{$5$}
      \psfrag{6}[][][0.5]{$6$}
       \psfrag{7}[][][0.5]{$7$}
\includegraphics[scale=0.5]{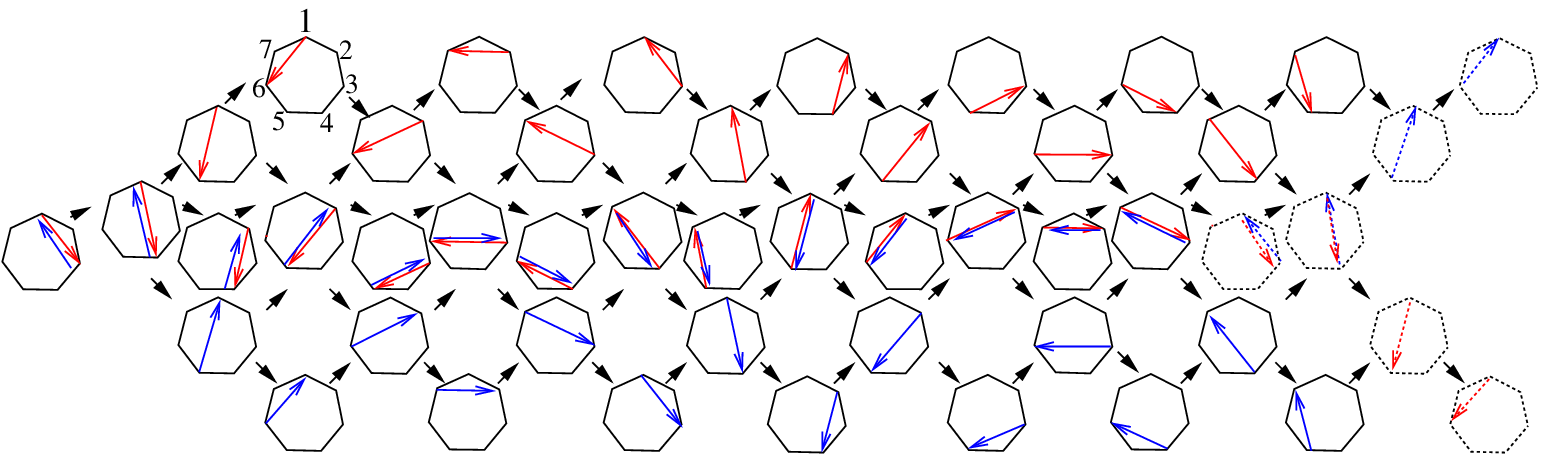}
\caption{The quiver $(\Gamma^7_{1,2,2},\tau)$
corresponds to the AR-quiver of
the orbit category $\mathcal{D}^b(\mathrm{mod}k E_6)/\tau^{-1}\Sigma$. }
\label{ARE6}
\end{figure}

\section{Equivalences of categories}\label{eqcat}

In this  section we show that
the construction of $\Gamma^{n+3}_{r,t,t}$ and $\Gamma^{n+3}_{r,s,t}$ 
allow us to model geometrically properties of a number of additive categories.

Let $\rho$ be the automorphism of $T_{r,t,t}$ and let $n\geq\max\{r+t+1,r+s+1\}$. 
Then we can show the following opening result.
\begin{lemma}\label{propT}
The quiver $(\Gamma^{n+3}_{r,s,t},\tau)$ is a stable translation quiver.
\end{lemma} 
\begin{proof} 
We have to show three things. 
First, that $\Gamma^{n+3}_{r,s,t}$ 
is connected, has no loops, and is locally finite. 
Second, that for every vertex $v$ of $\Gamma^{n+3}_{r,s,t}$  
the number of arrows going to $v$ equals the number of arrows leaving $v$. 
Third, that the map $\tau$ is bijective.

All the above properties follow from the construction. Let
us first consider the symmetric case.
Then it is not hard to see that when $\Pi$ is even sided,  
$\Gamma^{n+3}_{r,t,t}|_R\cong \mathbb Z A_{r+t+1}/\tau^{-(n+3)}$ 
and $\Gamma^{n+3}_{r,t,t}|_B\cong \mathbb Z A_{r+t+1}/\tau^{-(n+3)}$  since 
we consider oriented arcs. Thus, we deduce that
$(\Gamma^{n+3}_{r,t,t}|_R,\tau\vert_R)$ and $(\Gamma^{n+3}_{r,t,t}|_B,\tau\vert_B)$
are stable translation quivers.
Forming pairs of coloured oriented 
diagonals results in gluing these two 
quivers along $r$ disjoint $\tau$-orbits, and the above properties are preserved.
When $\Pi$ is odd sided one can check 
that the modifications in the definition 
of minimal clockwise rotations and in the definition of $\tau$
are such that the resulting quiver has the claimed properties. 

Finally, it is not hard to see that also $\Gamma^{n+3}_{r,s,t}$ 
associated to an asymmetric tree is a stable translation quiver. 
\end{proof}

\begin{thm}\label{isoar}
Let $\mathcal{T}$ be an additively finite Krull-Schmidt category. Let $\Gamma$
be a connected component of the AR-quiver of $\mathcal{T}$. Assume that
$\Gamma\cong \mathbb Z T_{r,s,t}/\tau^{-(n+3)}$, resp.
$\Gamma\cong\mathbb Z T_{r,t,t}/\tau^{-(n+3)}\rho$.
Then there is an isomorphism of stable translation quivers
$\Gamma^{n+3}_{r,s,t}\rightarrow \Gamma,$ 
for $\Gamma^{n+3}_{r,s,t}$ associated to a regular $(n+3)$-gon.
\end{thm}
\begin{proof}
The claim follows from the proof of Lemma \ref{propT}, since
we saw that the following are 
isomorphisms of stable translation quivers:
$\Gamma^{n+3}_{r,s,t}\stackrel{\simeq}{\rightarrow} \mathbb Z T_{r,s,t}/\tau^{-(n+3)}$,
$\Gamma^{n+3}_{r,t,t}\stackrel{\simeq}{\rightarrow} \mathbb Z T_{r,t,t}/\tau^{-(n+3)}\rho^{(n+3)}$
for all integers $r,s,t$ and $n$ as above.
\end{proof}

\subsection{Projections}\label{Projections}

Our next goal is to define a translation quiver $\Gamma^{n+3}_{r,t}$ obtained 
from ${\Gamma}^{n+3}_{r,t,t}$ associated to a symmetric tree $T_{r,t,t}$
after folding ${\Gamma}^{n+3}_{r,t,t}$
along its central line.

For this consider again the graph automorphism $\rho$ induced by 
the map  simultaneously changing colour and orientation of the diagonals in $\Pi_{r,t,t}$.
Then the vertices of $\Gamma^{n+3}_{r,t}$ are the 
$\rho$-orbits of vertices of $\Gamma^{n+3}_{r,t,t}$, i.e. the pairs
$\{ [i,i+j]_R, [i+j,i]_B \}$, for $i,j$ in vertices of $\Pi$.
The arrows in $\Gamma^{n+3}_{r,t}$ are always single and 
coincide with minimal clockwise rotation around a common vertex of $\Pi$
linking pairs of coloured oriented diagonals. 
The translation on $\Gamma^{n+3}_{r,t}$ is induced from the translation in $\Gamma^{n+3}_{r,t,t}$
and given by the anti clockwise rotation through $\frac{2\pi}{n+3}$ 
around the center of $\Pi$. 

Clearly, $\Gamma^{n+3}_{r,t,t}\rightarrow \Gamma^{n+3}_{r,t}$ is a surjective map 
of stable translation quivers.

\subsection{Cluster categories of type $E_6$, $E_7$ and $E_8$}

As a corollary of Theorem \ref{isoar} we obtain the geometrical modelling of cluster
categories $\mathcal{C}_{T_{r,s,t}}$ where $T_{r,s,t}$ is an arbitrary tree.
Since the number of connected components of the AR-quiver of $\mathcal{C}_{T_{r,t,t}}$ 
varies with the shape of $T_{r,t,t}$, we proceed considering two cases. In this section
we focus on tree graphs of Dynkin type, in Section \ref{gen} the general case will be treated.

In view of Theorem \ref{isoE6} below let
$\mathcal{C}^{n+3}_{r,s,t}$ be the additive category generated by the
mesh category of $\Gamma^{n+3}_{r,s,t}$, for $r,s,t,n\in\N$.
Then we can show the main result of this section.

\begin{thm}\label{eqofcat}\label{isoE6} 
We have the following equivalences of additive categories
\begin{align*}
&\mathcal{C}^7_{1,2,2}\rightarrow \mathcal{C}_{E_6} \\
&\mathcal{C}^{10}_{1,2,3}\rightarrow \mathcal{C}_{E_7}\\
&\mathcal{C}^{16}_{1,2,4}\rightarrow \mathcal{C}_{E_8}.
\end{align*}
\end{thm}

\begin{proof}
Since the full subcategories of indecomposable objects
of the orbit categories we consider are equivalent to the 
mesh category of their AR-quiver, we only have to check that
there is an isomorphism of stable translation quivers
between the AR-quiver of the various orbit 
categories and the quivers of coloured oriented single 
and paired diagonals associated
to $\Pi$.
This isomorphism then
induces the claimed equivalences. 

To do so, the strategy will be to compare the action of 
$\tau$ on $\Gamma^{n+3}_{r,s,t}$ with 
the actions of $\tau$ and of $\Sigma$ on $\mathbb Z Q$
where $Q$ is an orientation of a simply laced Dynkin diagram, 
as described in Section \ref{action}.

Below we treat the case $\mathcal{C}_{E_7}$, the remaining two claims 
can be deduced with a similar reasoning.
From the discussion of Section \ref{action} it follows that the 
AR-quiver of $\mathcal{C}_{E_7}$ is isomorphic to 
the quotient graph $\Z E_7/ \tau^{-10}$. On the other side, by definition
$\mathcal{C}^{10}_{1,5,5}$
is the mesh category of 
$\Gamma^{10}_{1,5,5}$ associated
to a 10-gon and  
$\Gamma^{10}_{1,5,5}\cong\Z E_7/ (\tau^{-10})\cong\tau^{-1}\Sigma$.
\end{proof}

\begin{cor}\label{eq}
We also deduce the following equivalences of additive categories
\begin{align*}
&\mathcal{C}^{r+4}_{r,0,0}\rightarrow \mathcal{C}^{2}_{A_{r+1}} \\
&\mathcal{C}^{r+s+4}_{r,s,0}\rightarrow \mathcal{C}^{2}_{A_{r+s+1}} \\
&\mathcal{C}^{r+t+4}_{r,0,t}\rightarrow \mathcal{C}^{2}_{A_{r+t+1}} \\
&\mathcal{C}^{r+5}_{r,1,1}\rightarrow \mathcal{D}^{b}(\mathrm{mod}D_{r+3})/\tau^{-3}\Sigma. \\
\end{align*}
\end{cor}
\begin{proof}
We follow the proof of Theorem \ref{eqofcat} and observe
that for the first claim
we consider only paired oriented diagonals of
$\Pi$. Since $[i,j]_P\neq[j,i]_P$ and the map $\rho$ is the identity on
paired oriented diagonals, we deduce that $\Gamma^{r+4}_{r,0,0}$ 
always lies on a cylinder.  From Lemma \ref{propT}
we deduce that $\Gamma^{r+4}_{r,0,0}\cong \mathbb Z A_{r+1}/\tau^{-(r+3)}$. 
From the discussion of Section \ref{action} we deduce that 
$\mathbb Z A_{r+1}/\tau^{-(r+3)}\cong\mathbb Z A_{r+1}/\tau^{-2}\Sigma$
which we recognise as the AR-quiver of $\mathcal{D}^b(\mathrm{ mod } k A_{r+1})/\tau^{-2}\Sigma^2$.

The second and third claim follow in a similar fashion. 

For the last claim we observe that
for each vertex of $\Pi$ the quiver $\Gamma^{r+5}_{r,1,1}$ 
has one red and one blue single coloured oriented diagonal. These
correspond to the exceptional vertices
of a Dynkin diagram of type $D_{r+3}$.
\end{proof}

\begin{rem}
The equivalences of Theorem \ref{eqofcat} and Corollary \ref{eq} 
allow us to define a shift functor $\Sigma$ on
the categories associated to $\Pi$
induced by the shift functor
$\Sigma$ defined on the various orbit categories considered above, see also
\cite{Lisa1}. 
In the following however, 
we will not use the triangulated structure of these categories. 
\end{rem}

Let $[i,j]_c$, $c\in\{R,B,P\}$, be a coloured oriented single or paired 
diagonal of a heptagon $\Pi$ and let $(i,j)$ be 
the underlying unoriented diagonal. 
Then we deduce the following useful results.

\begin{cor}\label{corproj}
There is a dense and full functor
\begin{align*}
  \mathcal{C}_{E_6}&\rightarrow\mathcal{C}_{A_{4}}\\
  [i,j]_c&\mapsto(i,j)
\end{align*}
from the cluster category $\mathcal{C}_{E_6}$ of type $E_6$ to 
the cluster category $\mathcal{C}_{A_4}$ of type $A_4$.
\end{cor}
\begin{proof}
Consider the projection
$\pi_1:\Gamma^7_{1,2,2}\rightarrow \Gamma^7_{1,2}$ 
defined in Subsection \ref{Projections}. Let 
$\Gamma$ be the stable translation quiver
of unoriented diagonals of $\Pi$, as defined in Caldero-Chapoton-Schiffler's paper \cite{CCS}.
Then there is a projection
$\pi_2:\Gamma^7_{1,2}\rightarrow\Gamma$,
which maps $\{ [i,i+j]_R, [i+j,i]_B) \}$
to the unoriented diagonal $(i,i+j)$ of $\Pi$
and pairs of  arrows in $\Gamma^7_{1,2}$ to 
the one corresponding arrow in $\Gamma$.
We get a surjective map of translation quivers
$\pi_2\circ\pi_1:\Gamma^7_{1,2,2}\rightarrow\Gamma.$
This map then induces a dense and full functor
$$\mathcal{C}_{E_6}\rightarrow \mathcal{C}_{A_4},$$
after identifying
$\mathcal{C}_{E_6}$, resp. $ \mathcal{C}_{A_4}$, 
with the additive category generated by
the mesh category of $\Gamma^7_{1,2,2}$, resp. $\Gamma$.
\end{proof}
In Section \ref{cluster e6}
we will use the functor
$\mathcal{C}_{E_6}\rightarrow\mathcal{C}_{A_{4}}$ to describe 
all cluster tilting sets
of $\mathcal{C}_{E_6}$ as configurations of coloured 
oriented diagonals in $\Pi$.

\begin{cor}
There is a dense and full functor
\begin{align*}
  \mathcal{C}^{r+4}_{r,0,0}&\rightarrow\mathcal{C}_{A_{r+1}}\\
  [i,j]_P&\mapsto(i,j)
\end{align*} from the category $\mathcal{C}^{r+4}_{r,0,0}$ to 
the cluster category of type $A_{r+1}$.
\end{cor}
\begin{proof}
Follows from the proof of Corollary \ref{eq}.
\end{proof}

\subsection{Coloured oriented single and paired diagonals
and the cluster category of type $D_n$}

The aim of this section is to link our category of coloured oriented single and paired diagonals
to the cluster category of type $D_n$, denoted by $\mathcal{C}_{D_n}$.
We recall that $\mathcal{C}_{D_n}$ can be modelled 
geometrically in two equivalent ways. One approach 
arises categorifying the model given by Fomin-Zelevinsky
in \cite[Prop. 3.16]{FZ_y} defined in terms of unoriented 
diameters and pairs of diagonals in a regular $2n$-gon.  
The second approach uses tagged arcs in a once punctured disc, 
as described by Schiffler in \cite{S}. We can recover both descriptions from our model. \newline

For the first one, we proceed as follows.
Let $(\Gamma_{n-3,n-1,n-1}^{2n}, \tau)$ be as before and 
define the quiver $\Gamma_{D_n}$ having as vertices 
the (unordered) triples of centrally symmetric diagonals of $\Pi$:
\begin{align*}\itemsep0em
&\{[i,i+2]_P, [i+2+n,i+n]_R, [i+n,i+2+n]_B\},\\
&\{[i,i+3]_P, [i+3+n,i+n]_R, [i+n,i+3+n]_B\},\\
&\dots\\
&\{[i,i+n-1]_P, [i+2n-1,i+n]_R, [i+n,i+2n-1]_B\}
\end{align*}
together with the oriented single 
diagonals $[i,n+i]_R, [n+1,i]_B$ of $\Pi$, for $1\leq i\leq2n.$ 
Notice that  $[i,n+i]_R $ and $[n+1,i]_B$ are the central oriented
diagonals of  $\Pi.$  The arrows of $\Gamma_{D_n}$ are induced by the minimal clockwise rotations between diagonals of $\Pi_{n-3,n-1,n-1}$, similarly for the translation map. This allows us to obtain a surjective map of stable translation quivers $\Gamma_{n-3,n-1,n-1}^{2n}\rightarrow \Gamma_{D_n}$. Dropping the orientations of all coloured diagonals we also obtain the surjective map $\Gamma_{D_n}\rightarrow AR(\mathcal {C}_{D_{n}})$. It follows that there is a dense and full functor $$\mathcal{C}_{n-3,n-1,n-1}^{2n}\rightarrow\mathcal {C}_{D_{n}}$$ from the mesh category of $\Gamma_{n-3,n-1,n-1}^{2n}$, $\mathcal{C}_{n-3,n-1,n-1}^{2n}$, to the cluster category $\mathcal {C}_{D_{n}}.$ In addition, we observe that $\Gamma_{D_n}\cong\mathbb{Z}D_n/(\tau^{-1}\Sigma)^2\cong AR(\mathcal {C}^2_{D_{n}})$, where $\mathcal {C}^2_{D_{n}}$ is the 2-repetitive cluster category of type $D_n$ defined in Section \ref{def_orbitcat}.\newline 

To recover Schiffler's model we need to use a smaller punctured polygon and allow non-contractible loops. Defining $\Gamma^{n}_{n-3,1,1}$ and $\mathcal{C}^{n}_{n-3,1,1}$ as before we obtain an isomorphism of stable translation quivers $\Gamma_{n-3,1,1}^{n}\stackrel{\cong}{\rightarrow}AR(\mathcal{C}_{D_{n}})$ and the desired equivalence of categories: $\mathcal{C}^{n}_{n-3,1,1}\stackrel{\simeq}{\rightarrow}\mathcal{C}_{D_{n}}.$ In Figure \ref{ARC_d} we illustrate the quiver $\Gamma^4_{1,1,1}$ in a regular punctured square.\begin{figure}[H]
\psfragscanon
 \psfrag{1}[][][0.5]{$1$}
  \psfrag{2}[][][0.5]{$2$}
  \psfrag{3}[][][0.5]{$3$}
  \psfrag{4}[][][0.5]{$4$}
\includegraphics[scale=0.75]{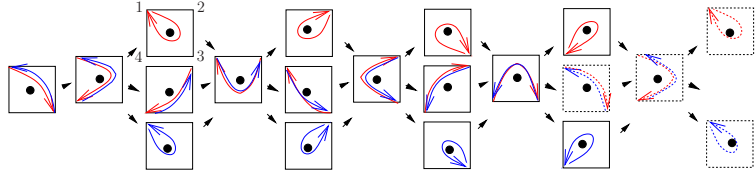}
\caption{The quiver $\Gamma^4_{1,1,1}\cong AR(\mathcal{C}_{D_{4}})$ in a square.}
\label{ARC_d}
\end{figure}

\subsection{Cluster categories associated to trees not of Dynkin type}\label{gen}

When $T_{r,s,t}$ is not an orientation of a simply laced Dynkin graph the AR-quiver 
of the cluster category $\mathcal{C}_{T_{r,s,t}}$ splits into various connected
components. 
Our next goal is to obtain an isomorphism between the irregular (transjecctive) component
of the AR-quiver of $\mathcal{C}_{T_{r,s,t}}$
and the quiver of coloured oriented single and paired diagonals.

For this we extend the construction of Section \ref{pairs} and consider
and infinite-sided polygon $\Pi^\infty$. Next, 
let $\Pi^\infty_{r,s,t}$ be the subset of all coloured oriented single and paired
diagonals of $\Pi^\infty$ consisting 
of $r+1$ paired, $s$ red and $t$ blue single coloured oriented
diagonals for every vertex of $\Pi^\infty$:
\begin{align*}\itemsep0em
\Pi_{r,s,t}:=\bigg\{&[i,i+2]_P, \dots,[i,i+r+2]_P, \\
&[i,i+r+3]_R,\dots,[i,i+r+s+2]_R, \\
&[i+r+3,i]_B, \dots,[i+r+t+2,i]_B,\hspace{0,3cm} \textrm{ $i$ vertex  of }\Pi^\infty\bigg\}.
\end{align*}

Then we extend the definition of minimal clockwise rotations of Section \ref{rotations} 
to this setting. Thus we obtain a quiver $\Gamma^\infty_{r,s,t}$,
whose vertices are the elements of $\Pi^\infty_{r,s,t}$ and where we
link two vertices with an arrow when 
there is a minimal clockwise rotation between them. 
Moreover, we define $\tau$ on the elements
of $\Pi^\infty_{r,s,t}$ as the anticlockwise rotation around the center of $\Pi^\infty$
induced by the rotation though $\frac{2\pi}{n+3}$ in an $(n+3)$-gon $\Pi$ letting $n\rightarrow\infty$.
In this way we turn $(\Gamma^\infty_{r,s,t},\tau)$ into a stable translation quiver.

To state the next result, let $\mathcal{P}$ be the full subcategory of $\mathcal{C}_{T_{r,s,t}}$,
consisting of $\tau$-shifts of indecomposable projective  objects in $\mathcal{C}_{T_{r,s,t}}$. 
Moreover, denote the additive
category generated by the mesh category
of $(\Gamma^\infty_{r,s,t},\tau)$ by $\mathcal{C}^\infty_{r,s,t}$.
Then we can show the following result.

\begin{prop}\label{isoinfty}
The functor 
$$\varphi:\mathcal{C}^\infty_{r,s,t}\rightarrow \mathcal{P}$$
is an equivalence of additive categories. 
\end{prop}
\begin{proof}
First we observe that $\Gamma^\infty_{r,s,t}\cong \Z T_{r,s,t}$.

Next, let $\Pi^\infty_{r,s,t}\vert_1$ be the subset of $\Pi^\infty_{r,s,t}$ consisting of the 
$r+s+t+1$ single and paired diagonals:
$\{[1,3]_P, \dots,[1,r+3]_P, 
[1,r+4]_R,\dots,[1,r+s+3]_R, 
[r+4,1]_B, \dots,[r+t+3,1]_B \}$.
Then, we observe that since $\Gamma^\infty_{r,s,t}$ is a connected stable translation quiver
$\varphi$ induces an injective map between $\mathrm{Hom}_{C^\infty_{r,s,t}}(\tau(D_j),D_i)$,
for $D_i,D_j\in \Pi^\infty_{r,s,t}$ and $\mathrm{Hom}_{C_{T_{r,s,t}}}(I_j,\Sigma(P_i))$
between the indecomposable injective objects
and $\Sigma$ of the projective objects in $\mathcal{C}_{T_{r,s,t}}$. Thus $\varphi$ is full.
Moreover, it is not hard to see that $\varphi$ is dense and faithful, thus
$\varphi$ is indeed an equivalence of additive categories.
\end{proof}

\section{Combinatorics of the cluster category $\mathcal{C}_{E_6}$}\label{cluster e6}

Our next aim is to describe the combinatorics of the cluster category
$\mathcal{C}_{E_6}\cong\mathcal{C}_{1,2,2}$ 
inside a heptagon. 

Throughout the chapter let $\Pi$ be 
a regular heptagon and let 
$(\Gamma,\tau):=(\Gamma^7_{1,2,2},\tau)$. Moreover, let
$\Gamma_{\Pi}$ be the stable translation quiver having
as vertices the unoriented diagonals of $\Pi$ and arrows given
by minimal clockwise rotations, see \cite{CCS} for details.

\subsection{Extension spaces in $\mathcal{C}_{E_6}$}\label{start_end}

The support of $\mathrm{Hom}(\tau^{-1}X,-)$ in 
$\mathcal{C}_Q$
is called the {\em front $\mathrm{Ext}$-hammock of $X$}. 
The support of $\mathrm{Hom}(-,\tau X)$ in $\mathcal{C}_Q$
is called the 
{\em back $\mathrm{Ext}$-hammock of $X$}.
These hammocks can be deduced from the AR-quiver
using mesh relations, or using starting and ending functions, 
see ~\cite[Chap. 8]{BMRRT}. For AR-quivers isomorphic to $\mathbb Z Q$ 
with $Q$ an orientation of a Dynkin graph, 
the support of $\mathrm{Hom}(-,-)$ has 
been described in detail in \cite{bongartz}.

Identify again $\mathcal{C}_{E_6}$, resp. $ \mathcal{C}_{A_4}$ with
the additive categories generated by the mesh categories of
$\Gamma$, resp. $\Gamma_{\Pi}$. In the sequel we view the 
hammocks inside $\Gamma$ or $\Gamma_{\Pi}$. Moreover, 
let $\mathrm{Ext}_{\Pi}^1(D_X,D_Y):=\mathrm{Ext}_{\mathcal{C}_{E_6}}^1(X,Y)$,
for coloured oriented diagonals $D_X,$ $D_Y$ in $\Pi$
and indecomposable objects
$X$ and $Y$ in $\mathcal{C}_{E_6}$ corresponding to $D_X$ and $D_Y$
by the equivalence
of Theorem \ref{isoE6}.
\begin{rem}\label{2cy}
As $\mathcal{C}_{E_6}$ is 2 Calabi-Yau:
$D\mathrm{Ext}_{\mathcal{C}_{E_6}}^1(X,Y)\cong\mathrm{Ext}^1_{\mathcal{C}_{E_6}}(X,Y), $
for all objects $X,Y$ in $\mathcal{C}_{E_6}$. Therefore, 
the back and front $\mathrm{Ext}$-hammocks in $\mathcal{C}_{E_6}$ coincide
for all objects. 
On the other hand, in $\mathcal{C}^2_{A_4}$ the back 
and front hammocks are 
disjoint, as the category is not $2$-Calabi-Yau. 
\end{rem}

\subsection{Lift of hammocks}\label{lift}
Consider again the projection
$\widetilde{\pi}:=\pi_2\circ\pi_1:\Gamma\rightarrow \Gamma_{\Pi}$ defined by
$D_{(i,j)}:=\widetilde{\pi}(D_X)$ as in Corollary \ref{corproj}.
For each $D_X$ we define 
two connected sub-quivers of $\Gamma$, $I_1(D_X)$ and $I_2(D_X)$, as follows.
The vertices of both $I_1(D_X)$ and $I_2(D_X)$, are the vertices of $\Gamma$ 
in the $\mathrm{Ext}$-hammock of $D_X$ in $\Gamma$ and in
the preimage under $\widetilde{\pi}$ of the $\mathrm{Ext}$-hammock 
of $D_{(i,j)}$ in $\Gamma_{\Pi}$. 
The arrows of $I_1(D_X)$ and $I_2(D_X)$
coincide with the arrows of $\Gamma$.
Then $I_1(D_X)$ 
contains the vertex $\tau^{-1}(D_X)$ and will be called the
{\em front crossing} of $D_X$, 
$I_2(D_X)$ contains 
$\tau(D_X)$ and will be called the {\em back crossing} of $D_X$.

Note that for all $D_X$, the sub-quivers $I_1(D_X)$ and $I_2(D_X)$ are disjoint.
In addition, all coloured oriented diagonals in $\Pi$ {\em crossing} $D_X$ in an interior point of $D_X$
are vertices of $I_1(D_X)\cup\rho(I_1(D_X))$ and $I_2(D_X)\cup\rho(I_2(D_X))$.
See Figure \ref{Exts}, were the vertices of inside the front and back crossings of $D_X$
are in heptagons with bold boundary, for $D_X$ a coloured oriented diagonal 
in the first slice of $\Gamma$.

\begin{figure}[H]
\begin{center}
\center
\psfragscanon
\psfrag{a}[][][0.75]{a)}
\psfrag{b}[][][0.75]{b)}
\psfrag{c}[][][0.75]{c)}
\psfrag{d}[][][0.75]{d)}
\psfrag{e}[][][0.75]{e)}
\psfrag{f}[][][0.75]{f)}
\includegraphics[scale=0.48]{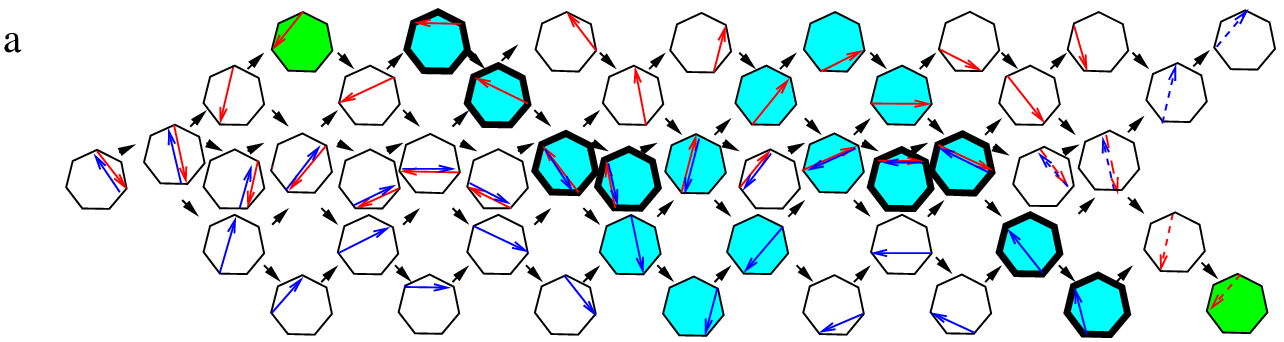}
\includegraphics[scale=0.48]{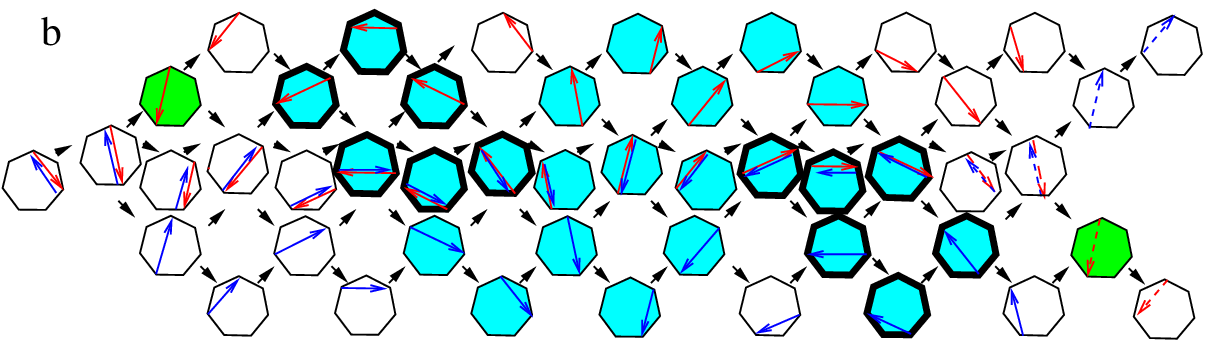}\vspace{0,4cm}
\includegraphics[scale=0.48]{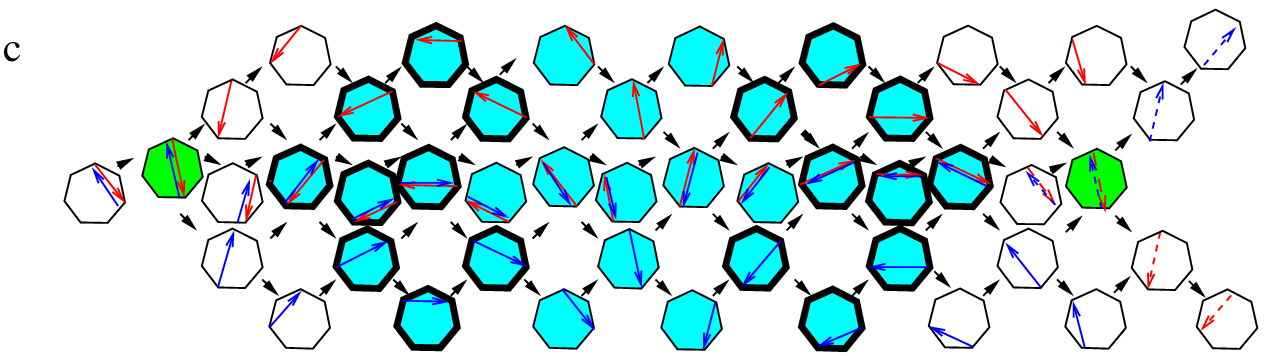}
\includegraphics[scale=0.48]{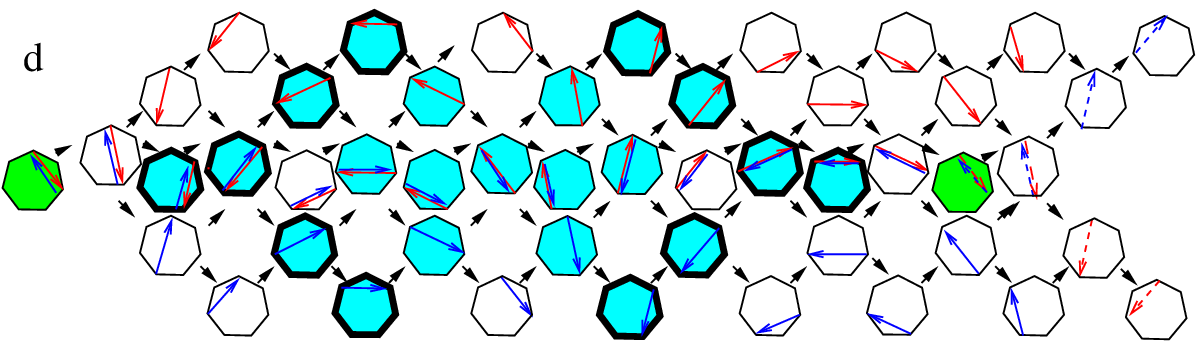}\vspace{0,4cm}
\includegraphics[scale=0.48]{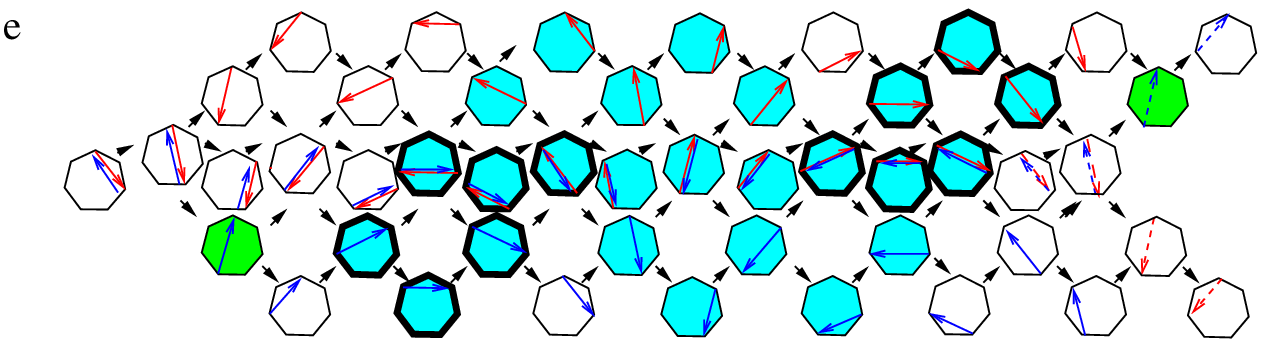}
\includegraphics[scale=0.48]{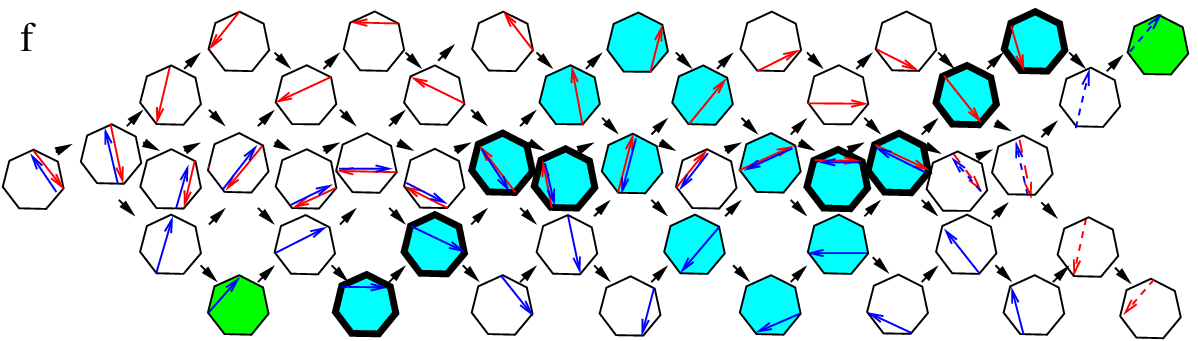}
\end{center}
\caption{ The $\mathrm{Ext}$-hammocks of the diagonals with 
vertices of $I_1$ and $I_2$ represented in 
heptagons with bold boundary.}
\label{Exts}
\end{figure}

In the next result, we assume that $D_X$ is a coloured oriented diagonal of
$\Pi$ in the first slice of $\Gamma$. This assumption
can be dropped using $\tau$-shifts, or 
renumbering the vertices of $\Pi$.
Moreover, we write $\partial\Pi$ to indicate the boundary of $\Pi$, 
and for two coloured
oriented diagonals $D_X$ and $D_Y$ we say that $D_Y$
{\em enters} the smaller region bounded by $D_{X}$ and $\partial\Pi$ 
if the arrow head of $D_Y$ goes to a vertex of $\partial\Pi$ inside the region
and different from the vertices joined by $D_{X}$.

\begin{prop}\label{cuts}
Let $D_{X}, D_{Y}$ be coloured oriented diagonals of $\Pi.$ 
Assume $D_X$ is in the first slice of $\Gamma$, and that $D_X$ crosses $D_{Y}$.
\begin{itemize}
 \item If $D_{X}$ is a paired diagonal, then
 $\mathrm{dim}_k(\mathrm{Ext}^1_{\Pi}(D_X,D_Y))=1$.
 \item If $D_{X}$ is a single diagonal,
 and $D_{Y}$ enters the 
 smaller region bounded by $D_{X}$ and $\partial\Pi$, then
 $\mathrm{dim}_k(\mathrm{Ext}^1_{\Pi}(D_X,D_Y))=1$.
\end{itemize}
\end{prop}
\begin{proof}
If $D_{X}$ is paired, $I_1(D_X)$ coincides with $I_1(D_X)\cup\rho(I_1(D_X))$ and
$I_2(D_X)$ coincides with $I_2(D_X)\cup\rho(I_2(D_X))$, thus the vertices of
$I_1(D_X)$ and
$I_2(D_X)$
are all the oriented coloured 
diagonals of $\Pi$ crossing $D_{X}$.

If $D_{X}$ is a single coloured diagonal of $\Pi$, we need to
distinguish between the diagonals
inside $I_i(D_X)$ and $\rho(I_i(D_X))$, $i=1,2$.
Then we observe that the coloured oriented diagonals in $I_1(D_X)$ and $I_2(D_X)$
are precisely the ones satisfying the 
assumptions of the proposition.
\end{proof}

\subsection{Curves of oriented coloured diagonals}\label{curves}
The aim of this section is to divide
the $\mathrm{Ext}$-hammocks
in $\mathcal{C}_{E_6}$  into curves. The reason why we do this is because 
for each coloured oriented
diagonal $D_X$
we want to find a uniform geometric description of the elements inside the
$\mathrm{Ext}$-hammock of $D_X$. 
Since the hammocks in $\mathcal{C}_{E_6}$ are very big,
this goal seems hopeless.
However, dividing the $\mathrm{Ext}$-hammock of $D_X$
into smaller sets, allows us to describe
the elements of each such
set in geometric terms.  We will call these sets curves.

Let $X$ be an indecomposable object of $\mathcal{C}_{E_6}$
and let $D_X$ be the corresponding coloured oriented diagonal viewed as a
vertex of $\Gamma$.
For $r\in\{2,4\}$, 
{\em the curves} $C_1(D_X),\ldots,C_r(D_X)$ of $D_X$
in $\Gamma$
are $r$ collections of oriented coloured diagonals 
having non-vanishing extensions with $D_X$.
Each collection $C_i(D_X)$
has the shape of a curve in $\Gamma$. 

We begin defining the
curves of $D_X$, for $D_X$ in the first slice of $\Gamma$. 
For all other vertices $D_X$ of $\Gamma$, curves can be defined
from the previous ones by $\tau$-shifts.
The first curve of $[1,6]_R$, denoted by $C_1([1,6]_R)$, is defined as follows:
\begin{align*}
C_1([1,6]_R):=&\{[7,2+i] _c,\, 0\leq i\leq 3,\, c\in\{R,P\}\}\\
&\cup\{[5,7+i] _c,\, 0\leq i\leq 3,\,c\in\{R,P\} \}\\
&\cup\{[6,3]_R\}.
\end{align*}
The second curve of $[1,6]_R$, denoted by $C_2([1,6]_R)$, is defined as follows:
\begin{align*}
C_2([1,6]_R):=&\{[5-i,7] _c,\, 0\leq i\leq 3,\, c\in\{B,P\}\}\\
&\cup
\{[2,7+i] _c,\, 0\leq i\leq 3,\,c\in\{B,P\} \}\\
&\cup\{[1,4]_B\}.
\end{align*}
By definition $C_1([1,6]_R)$ is obtained by the 
sequence of minimal clockwise rotations around the vertices 7, 5, 3 of $\Pi$
starting in
$\tau^{-1}([1,6]_R)=[2,7]_R$ and ending with
$[6,3]_R$. Dually, $C_2([1,6]_R)$ is obtained by a sequence
of minimal anticlockwise rotations around the vertices 7, 2, 4  
starting in $\tau([1,6]_R)=[5,7]_B$ and
ending in $[1,4]_B$.

By construction the $\mathrm{Ext}$-hammock of $[1,6]_R$ 
is $C_1([1,6]_R)\cup C_2([1,6]_R)$. 
In Figure \ref{Hammocks}(a) the elements of $C_1([1,6]_R)$ 
are drawn in 
the upper half of $\Gamma$, while the elements
of $C_2([1,6]_R)$ are in the lower half. 

Next, we are going to associate four curves to $[1,5]_R$.  The first curve of $[1,5]_R$, 
$C_1([1,5]_R)$, is the set containing coloured oriented diagonals of $\Pi$
obtained by a sequence of minimal clockwise rotations around the vertices
6, 4, 2 starting in $\tau^{-1}([1,5]_R)$ and ending in $[5,2]_R$. 
The third curve $C_3([1,5]_R)$ is obtained by a 
sequence of minimal anticlockwise rotations around the vertices 7, 2, 4
starting with $\tau([1,5]_R)$ and ending in $[1,4]_B$.
Moreover, $C_2([1,5]_R)$ coincides with $C_1([1,6]_R)$, and $C_4([1,5]_R)$
coincides with $C_2([1,6]_R)$.

For $D_X\in\{[1,6]_R,[,1,5]_R\}$, the 
curves of $C_1(\rho(D_X)),\ldots,C_r(\rho(D_X))$ of
$\rho(D_X)$ are defined by $\rho(C_1(D_X)),\ldots,\rho(C_R(D_X))$, $r\in\{2,4\}$.

We are left with defining the curves of
the paired diagonals $[1,3]_P$ and $[1,4]_P$.
For $[1,3]_P$ we have $C_1([1,3]_P)$ given by the set containing
both single and paired coloured oriented diagonals
obtained by a sequence of minimal rotations 
in the clockwise order around the vertices 2, 7, 5
starting in $\tau^{-1}([1,3]_P)$ and ending in $[5,1]_P$.
Similarly $C_2([1,3]_P)$ is obtained by a sequence 
of minimal anticlockwise rotations around the vertices 2, 4, 5
starting in $\tau([1,3]_P)$ and ending in $[6,2]_P$.

Next, $C_1([1,4]_P)$ is 
obtained by a sequence of 
of minimal clockwise rotations starting in
$\tau^{-1}([1,4]_P) $ and ending in $[5,1]_P$.
$C_3([1,4]_P)$ is obtained rotating in the
anticlockwise order $\tau([1,4]_P)$ to $[4,7]_P$.
Finally, $C_2([1,4]_P)=\tau^{-1}(C_1([1,3]_P))$ and
$C_4([1,4]_P)=C_2([1,3]_P)$.

In Figure \ref{Hammocks}(a)-(f) we represent the various curves of $D_X$,
for$D_X$ be in the first (and last)
slice of $\Gamma$.
The numbers 1, 2, 3, 4 indicate
the starting term of the curves $C_1(D_X),\ldots,C_4(D_X)$ and
the colours of the heptagons indicate the curves they intersect.

\begin{figure}
\begin{center}
\center
\psfragscanon
\psfrag{a}[][][0.75]{(a)}
\psfrag{b}[][][0.75]{(b)}
\psfrag{c}[][][0.75]{(c)}
\psfrag{d}[][][0.75]{(d)}
\psfrag{e}[][][0.75]{(e)}
\psfrag{1}[][][0.75]{\textbf{1}}
\psfrag{2}[][][0.75]{\textbf{2}}
\psfrag{3}[][][0.75]{\textbf{3}}
\psfrag{4}[][][0.75]{\textbf{4}}
\psfrag{f}[][][0.75]{(f)}
\centering
\includegraphics[scale=0.48]{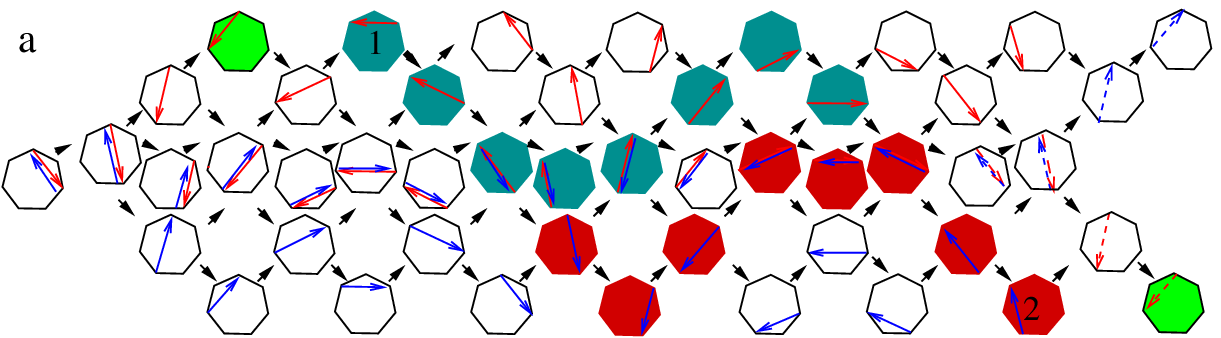}
\includegraphics[scale=0.48]{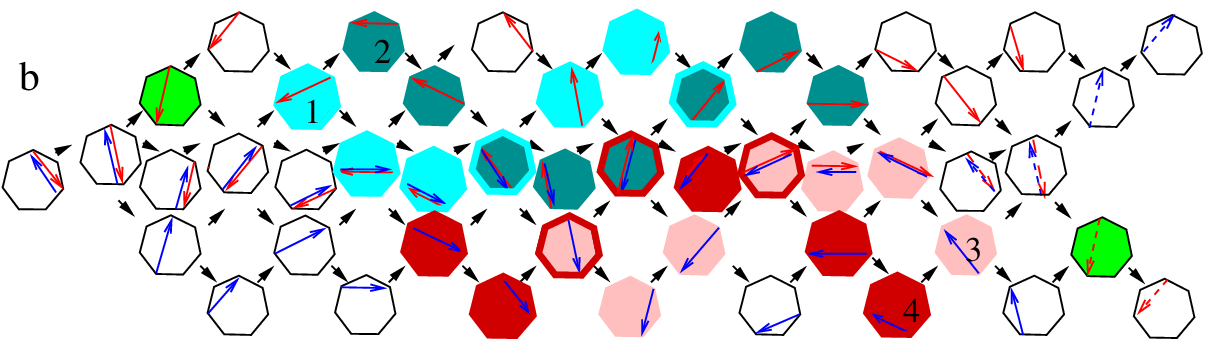}\vspace{0,4cm}
\includegraphics[scale=0.48]{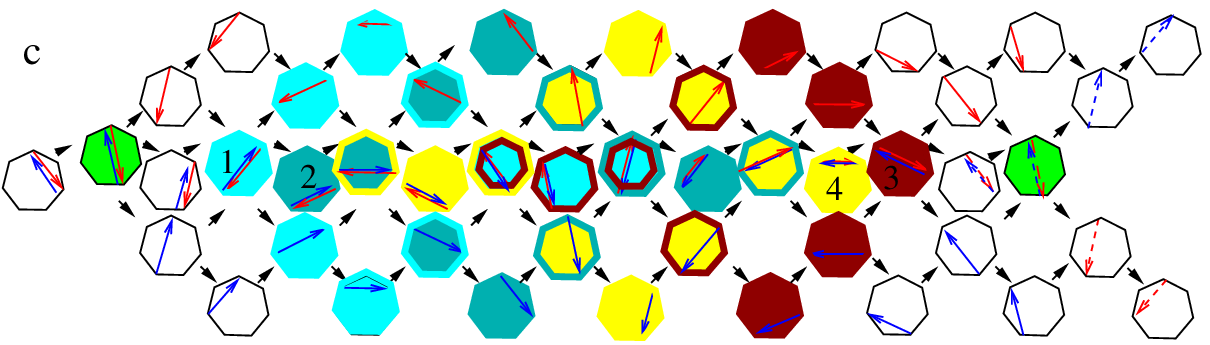}
\includegraphics[scale=0.48]{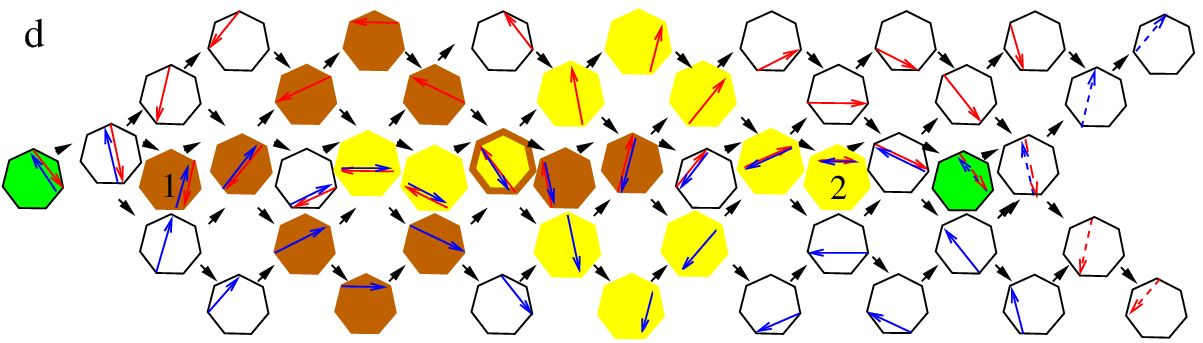}\vspace{0,4cm}
\includegraphics[scale=0.48]{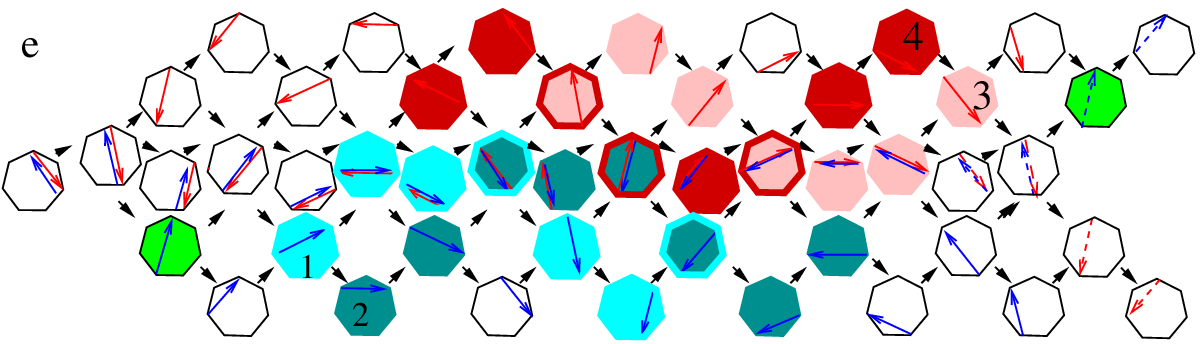}
\includegraphics[scale=0.48]{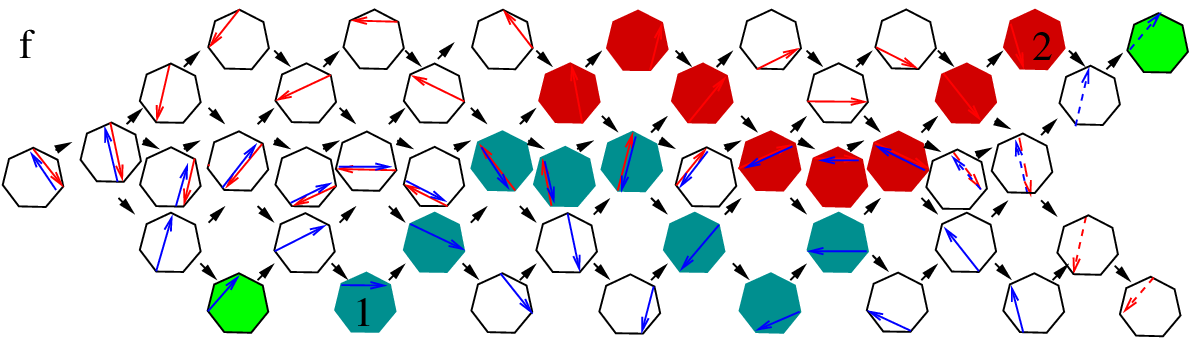}
\end{center}
\caption{Decomposition into curves of the Ext-hammocks of the vertices in the 
first slice of $\Gamma$.}
\label{Hammocks}
\end{figure}

\subsection{Intersections of curves}

\begin{prop} \label{int}
Let $D_X,D_Y$ be vertices of $\Gamma$. Let $C_1(D_X),\ldots,C_r(D_X)$
be the curves of $D_X$, $r \in\{2,4\}$. Then
$\mathrm{dim}_k(\mathrm{Ext}^1_{\Pi}(D_X,D_Y))$ is equal to
the number of curves of $D_X$ intersecting
with $D_Y$ in $\Gamma$ (0 up to 3).
\end{prop}
\begin{proof} First, the $\mathrm{Ext}$-hammocks in 
the AR-quiver of $\mathcal{C}_{E_6}$
are invariant under $\tau$-shifts. After changing the
image of the projective objects of $\mathrm{mod} k E_6$
in the equivalence of Theorem \ref{isoE6} we
can assume that  
$X$ or $Y$ corresponds to a diagonal 
in the first slice of $\Gamma$. 
By  remark \ref{2cy} we can treat the cases where $D_X$, or $D_Y$
belongs to the first slice
of $\Gamma$ in the same way.
Second, the curves of $D_X$ are by construction 
such that their intersection points 
coincide with the vertices $D_Y$ in $\Gamma$ for which 
$\mathrm{dim}_k(\mathrm{Ext}^1_{\Pi}(D_X,D_Y))\geq1$. 
From the definition of 
morphisms in the mesh category of a stable translation quiver,
it follows that the number of curves intersecting in $D_Y$ is 
the dimension of  $\mathrm{Ext}^1_{\mathcal{C}_{E_6}}(D_X,D_Y).$
\end{proof}

Let again $D_X$ be in the first slice of $\Gamma$.
In Figure \ref{Hammocks}(a)-(f)  
the dimension of the 
space $\mathrm{dim}_k(\mathrm{Ext}^1_{\Pi}(D_X,D_Y))$
is expressed by the numbers of colours filling the heptagon containing
$D_Y$. If $D_Y$ is in a white heptagon of $\Gamma$ then 
$\mathrm{dim}_k(\mathrm{Ext}^1_{\Pi}(D_X,D_Y))=0$.

More precisely, the two curves represented in 
Figure \ref{Hammocks}(a) and (f) 
never intersect, and $\mathrm{dim}_k(\mathrm{Ext}^1_{\Pi}(D_X,D_Y))=1$ for 
$D_Y$ in a coloured heptagon, and
$D_X$ in the first slice of $\Gamma$.

In Figure \ref{Hammocks}(b) and (e) 
the curve $C_1(D_X)$ intersects $C_2(D_X)$ in 
two vertices.
Similarly, for $C_3(D_X)$ and $C_4(D_X)$. 
The curve $C_2(D_X)$ intersects $C_4(D_X)$ only once.
The five heptagons where two curves meet have 
two colours (boundary and interior of the heptagon). 
Then $\mathrm{dim}_k(\mathrm{Ext}_{\Pi}(D_X,D_Y))=2$
for $D_Y$ corresponding to one of these heptagons.

In Figure \ref{Hammocks}(c) 
there are two heptagons where three curves meet. 
They 
are drawn with three colours, and 
hence $\mathrm{dim}_k(\mathrm{Ext}_{\Pi}(D_X,D_Y))=3$ 
for $D_Y$ corresponding to one of these two.  
Moreover, in nine heptagons two curves meet, 
and they are drawn in two colours.

Finally, in Figure \ref{Hammocks}(d) two curves are drawn, 
and they intersect only in one vertex of $\Gamma$.

\subsection{Cluster tilting objects}

Let $Q$ be an orientation of a simply-laced Dynkin 
graph with $n$ vertices.
Let $\mathcal{T}=\{T_1,T_2,\dots,T_n\}$ be a set of 
pairwise non isomorphic indecomposable objects of $\mathcal{C}_Q$.
If $\mathrm{Ext}_{\mathcal{C}_{Q}}^1(T_i, T_j)=0$ for all $T_i,T_j\in\mathcal{T}$, 
then one says that $\mathcal{T}$ is a {\em cluster tilting set} of $\mathcal{C}_{Q}$.
A {\em cluster tilting object} in $\mathcal{C}_{Q}$ is the direct sum of all objects
of a cluster tilting set in $\mathcal{C}_Q$. Observe that knowing
a cluster tilting objects allows to determines a cluster tilting set and viceversa.
Moreover, given a cluster tilting set $\mathcal{T}$, one says that 
$\overline{T}=\oplus_{j\neq i}T_j$, $T_j\in\mathcal{T}$ is an
{\em almost complete cluster tilting object} 
if there is an indecomposable object $T_i^*$ in $\mathcal{C}_Q$
such that $\overline{T}\oplus T_i^*$ 
is a cluster tilting object of $\mathcal{C}_Q$. The object $T_i^*$
is called the {\em complement} of 
$T_i$.

The {\em mutation at $i$} of a cluster tilting object 
$\mathcal{T}$ in $\mathcal{C}_Q$, for $1\leq i\leq n$, is the operation which
replaces the indecomposable summand $T_i$ in 
$\oplus_{j=1}^nT_j$ with the 
complement $T_i^*$ of $T_i$ in $\overline{T}=\oplus_{j\neq i}T_j.$

The statements in the next Theorem are shown in \cite{BMRRT}.
\begin{thm}\label{thmbmrrt}
Let $T$ be a cluster tilting object in $\mathcal{C}_{Q}$.
\begin{itemize}
\item each almost complete cluster tilting object $\overline{T}$ in $\mathcal{C}_{Q}$
has exactly two complements, $T$ and $T^*$.
\item If $T$ and $T^*$ are complements of $\overline{T}$ then
$\mathrm{dim}_k(\mathrm{Ext}_{\mathcal{C}_{Q}}^1(T,T^*))=1.$ On the other side,
if $\mathrm{dim}_k(\mathrm{Ext}_{\mathcal{C}_{Q}}^1(T,T^*))=1$, then there is an almost
complete cluster tilting object $\overline{T}$ such that $T$ and $T^*$ are complements of $\overline{T}$.
\end{itemize}
\end{thm}
After Proposition 3.8 in \cite{FZ_y} and \cite[Thm. 4.5]{BMRRT} 
we know that there are 833 cluster tilting sets, hence 
cluster tilting objects, in $\mathcal{C}_{E_6}$.

From the work of Caldero-Chapoton-Schiffler, see \cite{CCS}, we know 
that cluster tilting objects 
of $\mathcal{C}_{A_n}$ are in bijection 
with the formal direct sums of diagonals 
belonging to a maximal collection of non-crossing 
diagonals in a regular $(n+3)$-gon. In this context
mutations corresponds to {\em flips} of diagonals. 
More precisely, a flip replaces 
a diagonal $D_i$ in a given triangulation $\Delta$
with the unique other diagonal $D_i^*$ crossing
$D_i$ and completing $\Delta\backslash D_i$ to a new
triangulation of the regular $(n+3)$-gon.

\subsection{First fundamental family of 
cluster configurations of $\Pi$}\label{longpaired}

Our next aim is to describe cluster tilting sets
of $\mathcal{C}_{E_6}$ as configurations of single and paired
coloured oriented diagonals in $\Pi.$

\begin{de}
A {\em cluster configuration} is a family of pairwise different 
coloured oriented diagonals of $\Pi$, $\mathcal{T}=\{D_1,D_2,\ldots,D_6\}$,
with the property that
$\mathrm{Ext}_{\Pi}^1(D_i, D_j)=0$ 
for all $D_i,D_j\in\mathcal{T}$.
A coloured oriented diagonal $D_i^*$
is called  {\em complement} of 
$D_i$ in $\mathcal{T} $ if  $D_i^*\neq D_i$ and
$\mathcal{T}'$ obtained from $\mathcal{T} $ after replacing ${D_i}$ by ${D_i^*}$
is a cluster configuration of $\Pi$. 
\end{de}

Consider two heptagons, and a long paired diagonal 
$L_P:=[i,i+3]_P$, $i\in\mathbb Z/7\mathbb Z$ of $\Pi$.
Our next goal is to complete $L_P$ to a set of coloured oriented diagonals 
inside the two heptagons, giving rise to a cluster configuration of $\Pi$. 
For this we remark that $L_P$ divides each $\Pi$
into the quadrilateral $\Pi_4$ with boundary vertices 
$\{i, i+1,i+2,i+3\}$, and the pentagon $\Pi_5$ with boundary vertices
$\{i, i+3,i+4,i+5,i+6\}$, $i\in\mathbb Z/7\mathbb Z$.

In Lemma \ref{cor} below we will see
that triangulating $\Pi_4$ with a short paired diagonal, and  
each $\Pi_5$ with single diagonals of the appropriate colour
gives rise to cluster configurations.

\begin{lemma}\label{cor}
Let $L_P:=[i,i+3]_P$, $i\in\mathbb Z/7\mathbb Z$.
\begin{itemize} 
\item For $i\neq1$, triangulating each
$\Pi_5$ with single diagonals of the same colour, 
and $\Pi_4$ with a short paired diagonal
gives a cluster configuration $\mathcal{T}_{L_P}$ of $\Pi$. 
\item All cluster configurations of $\Pi$ containing $[j,j+3]_P$ 
arise as $\tau^k(\mathcal{T}_{L_P})$, $1\leq j,k\leq 7$.
\end{itemize}
\end{lemma} 

\begin{proof} 
Since $i\neq1$ we can assume that the region in $\Gamma$ 
outside the $\mathrm{Ext}$-hammock
of $L_P$ has only blue diagonals below $L_P$ and only red diagonals 
above $L_P$. Observes that the diagonals outside
the $\mathrm{Ext}$-hammock are precisely the diagonals involved in 
triangulations of the two copies of $\Pi_5$ and $\Pi_4$.

Then chose a short paired diagonal $S_P$ triangulating $\Pi_4$ with a short paired diagonal. 
Then triangulating
a copy of $\Pi_5$ with only single red diagonals, and triangulating
the second copy of $\Pi_5$ with only single blue diagonals yields a
cluster configuration.
With Proposition \ref{cuts} we deduce that the arcs 
obtained in this way have no extension in each region above and below $L$ in $\Gamma$.
One can then check that the $\mathrm{Ext}$-hammocks in one region
do not pass through the other region, nor though $S_P$.
Thus, to each red triangulation 
one can choose a blue triangulations of $\Pi_5$, and all choices are possible.
Similarly, one can complete $\{L_P, S_P^*\}$ to a cluster configuration, where 
$S_P^*$ is the other short paired diagonal triangulating $\Pi_4$.
Notice that there are no other possibilities to complete $L_P$ to a cluster configuration of $\Pi$.
Next, there are 7 choices for $L_P$ in $\Gamma$. For each choice of $L_p$ the associated
cluster configurations are obtained from the previous 
by rotation though $\tau$. Adjustment of the colours-orientations of the single 
diagonals triangulating $\Pi_5$ are needed if $L_P$ is the first slice of $\Gamma$.
\end{proof}
Notice that the two triangulations of $\Pi_5$ can be different, 
and the color is uniquely determined by the position of $L_P$
in $\Pi$, resp. in $\Gamma$.

In the following we call the cluster configurations
given by a long paired diagonal and coloured oriented single and paired
diagonals triangulating two copies of $\Pi_5$
and $\Pi_4$, as describe in the first part of Lemma \ref{cor},
the {\em first fundamental family} 
of cluster configurations.
We denote this family by $\mathcal{F}_1$.

\subsection{Second fundamental family of cluster configurations of $\Pi$}
We saw in Lemma \ref{cor} that many cluster 
configurations correspond
to two triangulations of $\Pi$. 
Our next goal is to define a second family of
cluster configurations describing the remaining
cluster tilting set of $\mathcal{C}_{E_6}$.
For this the following general observation is needed.

For $i\in\mathbb Z/7\mathbb Z$, consider the
long single red diagonal $L=[i,i+4]_R$ of $\Pi$. Then $L$
divides $\Pi$
into the quadrilateral $\Pi_4:= 
\{ i+4,i+5,i+6,i\}$, and the pentagon
$\Pi_5:=\{ i, i+1,i+2,i+3,i+4\}$. 
Let $\mathcal{T}_L$ be a cluster configuration of $\Pi$
containing $L$. Then $\mathcal{T}_L$ necessarily also contains one of 
the two short single diagonals triangulating $\Pi_4$,
neighbouring $L$ in $\Gamma.$ Similarly for 
$\rho(L)=[i+4,i]_B$.
More precisely, 
\begin{lemma} \label{l1} Let $i\in\mathbb Z/7\mathbb Z$,
$L=[i,i+4]_R$ in $\Pi$, and $\mathcal{T}_L$ be
a cluster configuration containing $L$.
\begin{itemize}
\item If $i\neq1$, exactly one of $\{[i,i+5]_R,[i+6,i+4]_R\}$ is in $\mathcal{T}_L$.
\item If $i=1$, exactly one of 
$\{[i,i+5]_R,\rho([i+6,i+4]_R)\}$ is in $\mathcal{T}_L$. 
\end{itemize}
Similarly for $\rho(L)$.
Moreover, in each case the two diagonals are complements to each other.
\end{lemma}
\begin{proof}
Let $i\neq1$ and consider the $\mathrm{Ext}$-hammock of $L$ in $\Gamma$. 
Since $L$ is not in the first slice of $\Gamma$ the diagonals triangulating $\Pi_4$ have the same
color as $L$. Then one can check that 
all $\mathrm{Ext}$-hammocks of objects outside 
the $\mathrm{Ext}$-hammock of $L$, which are different from $[i,i-2]_R$ and
$[i-1,i-3]_R$, never contain single diagonals inside the quadrilateral $\Pi_4$ in $\Pi.$
Thus, by maximality we deduce that all cluster tilting 
sets containing $L$ necessarily also contain 
one of the diagonals inside $\Pi_4$.
Taking one diagonal triangulating $\Pi_4$ rules out the other, thus the two 
single diagonals triangulating
$\Pi_4$ are complements to each other.
For $i=1$, $L$ is in the first slice of $\Gamma$.  Then one can proceed
as before adjusting the colour of the diagonal triangulating $\Pi_4$.
\end{proof}

In view of the next result, we point out that
the short single diagonals of Lemma \ref{l1}, 
triangulating $\Pi_4$ and neighbouring $L$ in $\Gamma$,
are displayed in filled light grey heptagons in Figure \ref{tilt}(a)-(n). 

\begin{lemma}\label{tilt_fig}
Every six-tuple of diagonals of Figure \ref{tilt}(a)-(n) 
determines a cluster configuration of $\Pi$. 
\end{lemma}
\begin{proof}
For each choice of a short single diagonal of Lemma \ref{l1}
triangulating $\Pi_4$ and neighbouring $L$ in $\Gamma$,
the claim can be verified by checking 
that the diagonals in the highlighted heptagons 
have no extension among each other.
\end{proof}
In the following, we refer to the collection of cluster configurations of
 Figure \ref{tilt}(a)-(n) as the
{\em second fundamental family} of  cluster configurations of $\Pi$, and we 
denote this family by $\mathcal{F}_2$.

\begin{figure}[H]
\centering
\psfragscanon
\psfrag{a}[][][0.75]{(a)}
\psfrag{b}[][][0.75]{(b)}
\psfrag{c}[][][0.75]{(c)}
\psfrag{d}[][][0.75]{(d)}
\psfrag{e}[][][0.75]{(e)}
\psfrag{f}[][][0.75]{(f)}
\psfrag{g}[][][0.75]{(g)}
\psfrag{h}[][][0.75]{(h)}
\psfrag{i}[][][0.75]{(i)}
\psfrag{l}[][][0.75]{(l)}
\psfrag{m}[][][0.75]{(m)}
\psfrag{n}[][][0.75]{(n)}

\begin{subfigure}[b]{0.42\textwidth}
\includegraphics[width=\textwidth]{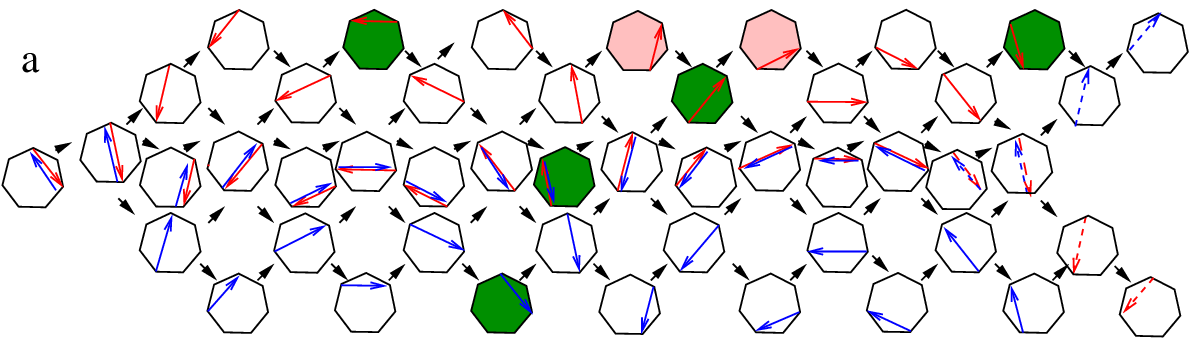}
\end{subfigure}\qquad 
\begin{subfigure}[b]{0.42\textwidth}
\includegraphics[width=\textwidth]{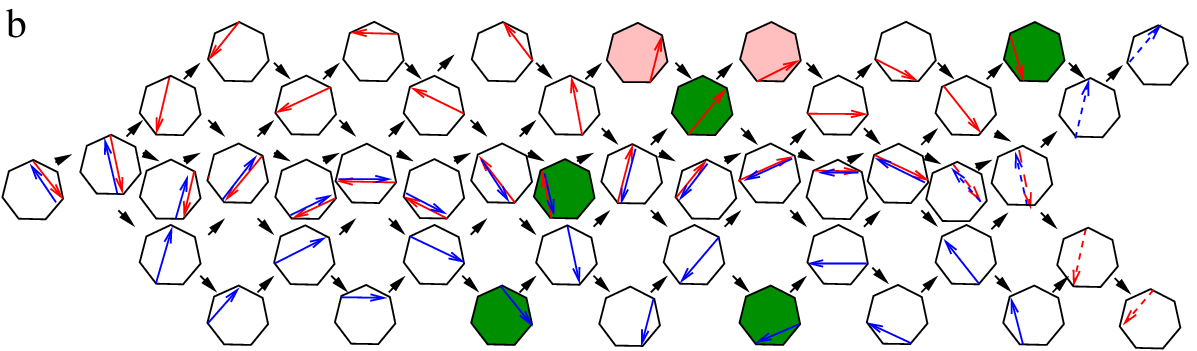}
\end{subfigure}\qquad   
\begin{subfigure}[b]{0.42\textwidth}
\includegraphics[width=\textwidth]{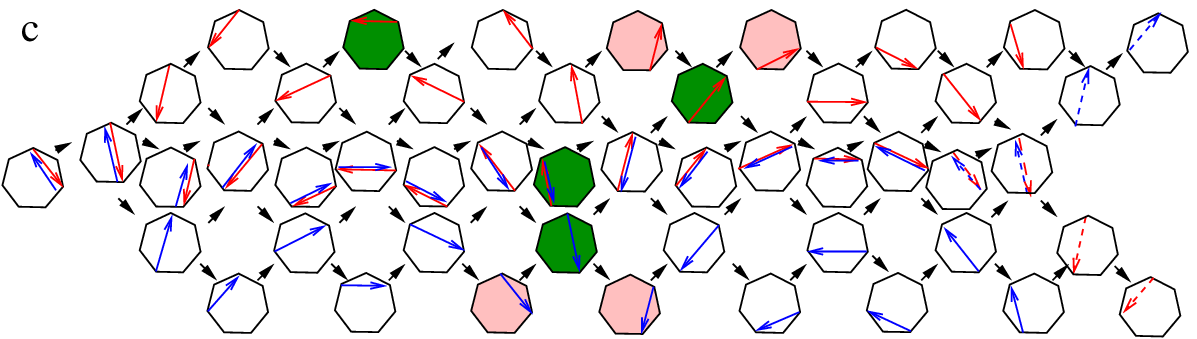}
\end{subfigure}\qquad 
\begin{subfigure}[b]{0.42\textwidth}
\includegraphics[width=\textwidth]{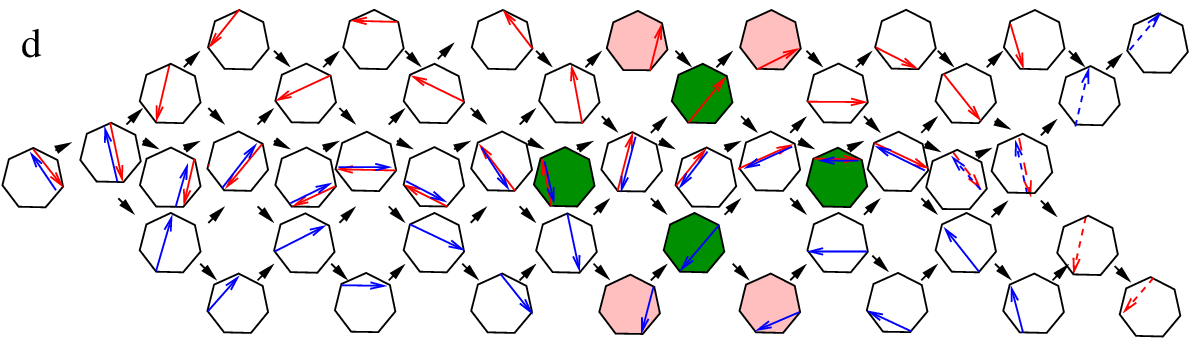}
\end{subfigure}\qquad  
\begin{subfigure}[b]{0.42\textwidth}
\includegraphics[width=\textwidth]{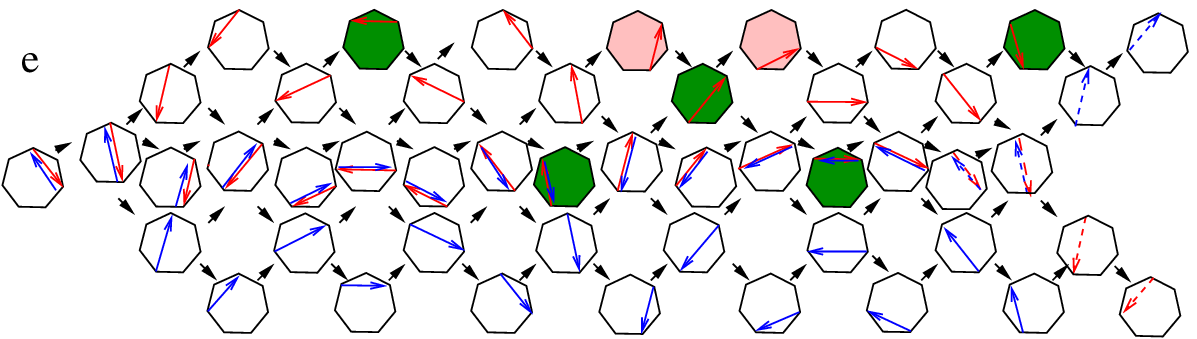}
\end{subfigure}\qquad 
\begin{subfigure}[b]{0.42\textwidth}
\includegraphics[width=\textwidth]{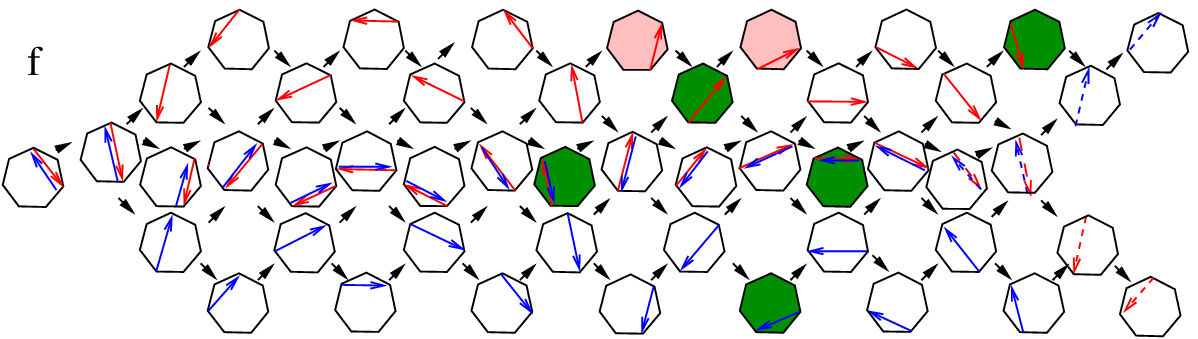}
\end{subfigure}\qquad 
\begin{subfigure}[b]{0.42\textwidth}
\includegraphics[width=\textwidth]{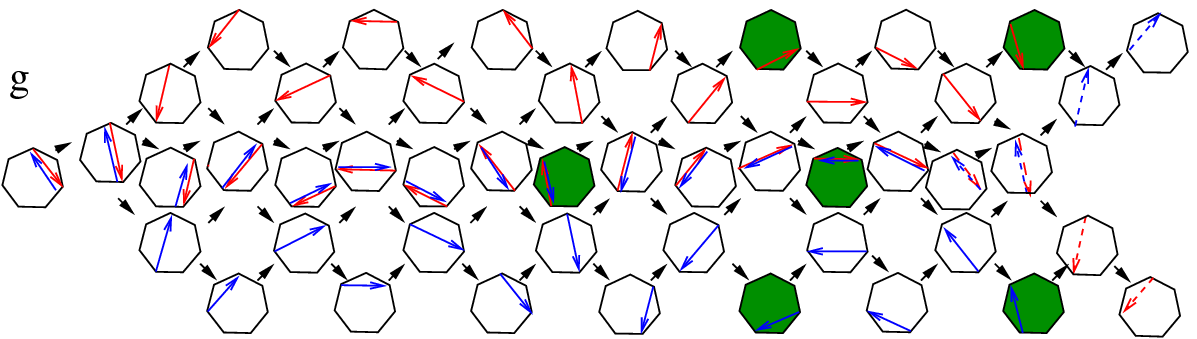}
\end{subfigure}\qquad 
\begin{subfigure}[b]{0.42\textwidth}
\includegraphics[width=\textwidth]{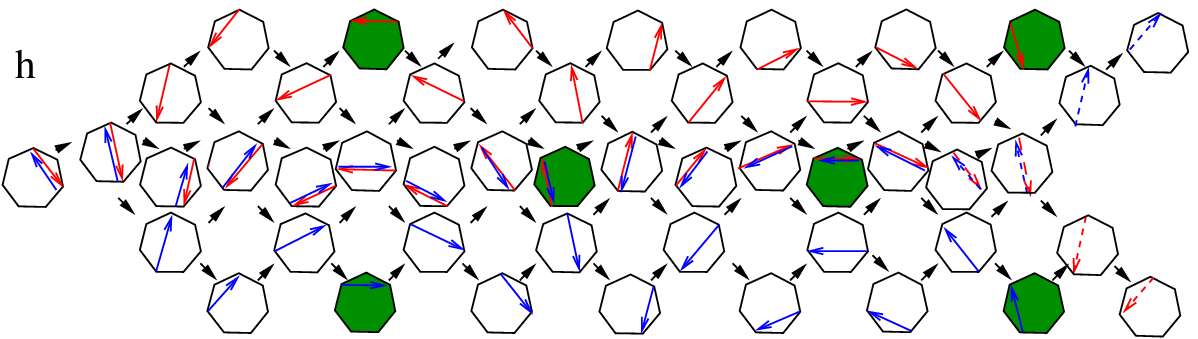}
\end{subfigure}\qquad 
\begin{subfigure}[b]{0.42\textwidth}
\includegraphics[width=\textwidth]{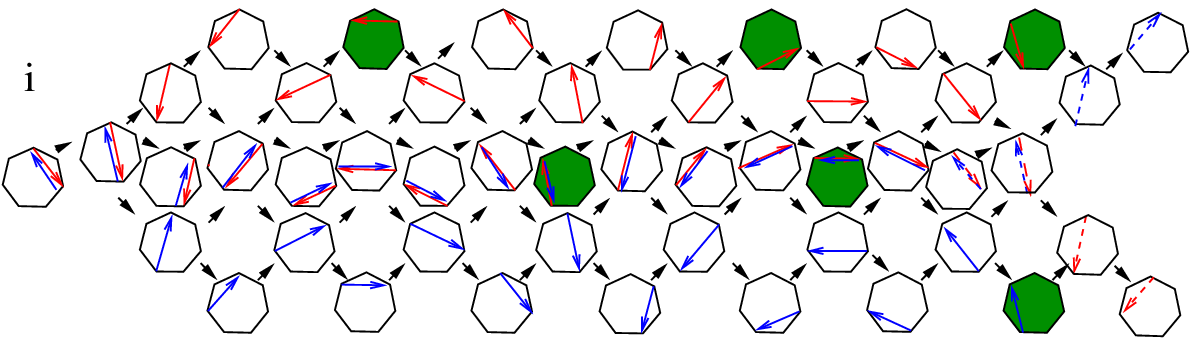}
\end{subfigure}\qquad 
\begin{subfigure}[b]{0.42\textwidth}
\includegraphics[width=\textwidth]{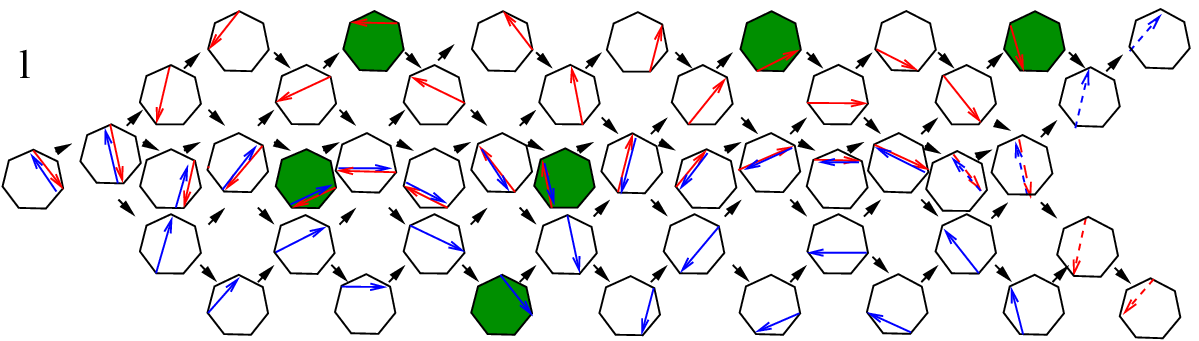}
\end{subfigure}\qquad 
\begin{subfigure}[b]{0.42\textwidth}
\includegraphics[width=\textwidth]{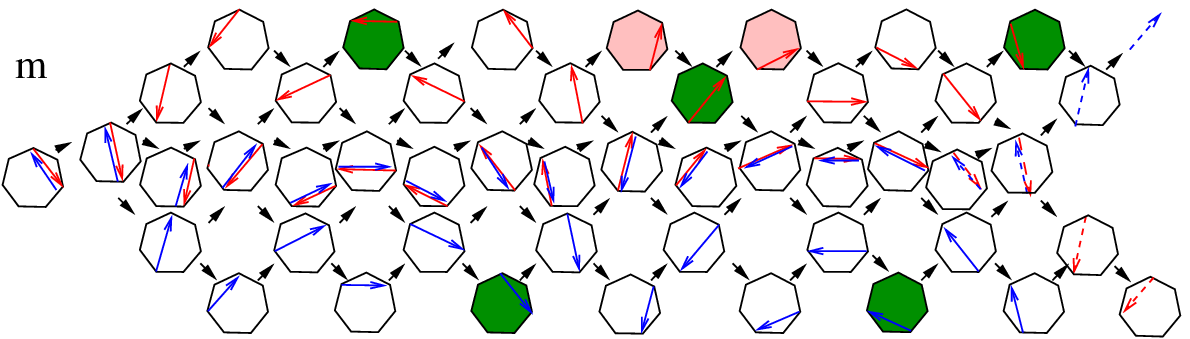}
\end{subfigure}\qquad 
\begin{subfigure}[b]{0.42\textwidth}
\includegraphics[width=\textwidth]{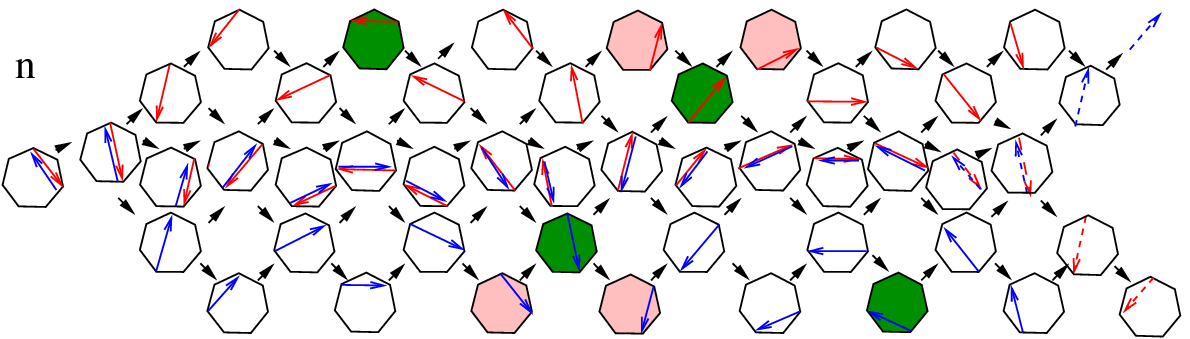}
\end{subfigure}\qquad 
\caption{The second fundamental family of cluster configurations of $\Pi$.}
\label{tilt}
\end{figure}

\subsection{Symmetries in $\Pi$ leading to cluster 
configurations}\label{sigma}

We determine two symmetries in $\Pi$ leading to 
cluster configurations.
One symmetry simply switches colours and orientations of 
the coloured oriented diagonals of a given cluster configuration.
The second one
arises from a left-right symmetry of $\Gamma$, and 
corresponds to a reflection in $\Pi$.

There are two reasons why these symmetries are important. 
First, using these symmetries 
we can deduce all cluster configurations starting from the 
sets in $\mathcal{F}_2$.
Second, knowing how a cluster configuration behaves under 
mutation, allows to understand 
how the symmetric ones behave. 

Let $c\in\{R,B\}$. For $i\in\mathbb Z/7\mathbb Z$ let
$h_i$ be the line in $\Pi$ passing through $i$ and the 
middle point of $i+3,i+4$.
On all coloured oriented
diagonals different then $[i\pm1,i\mp1]_c$, let 
$\sigma_i:\Pi\rightarrow\Pi$ be the reflection in $\Pi$ along 
$h_i$ followed by a switch of orientation. Otherwise,
$\sigma_i([i\pm1,i\mp1]_c):=\rho([i\pm1,i\mp1]_c).$

In Figure \ref{g'} we illustrate on the left the cluster configuration (g) 
and on the right we show the cluster configuration obtained 
after applying $\sigma_6$ to it.
 
\begin{figure}[H]
\centering
\psfragscanon
\psfrag{g}[][][0.75]{(g)}
\psfrag{h}[][][0.75]{(g')}
\begin{subfigure}[b]{0.42\textwidth}
\includegraphics[width=\textwidth]{tilt_f.eps}
\end{subfigure}\qquad 
\begin{subfigure}[b]{0.42\textwidth}
\includegraphics[width=\textwidth]{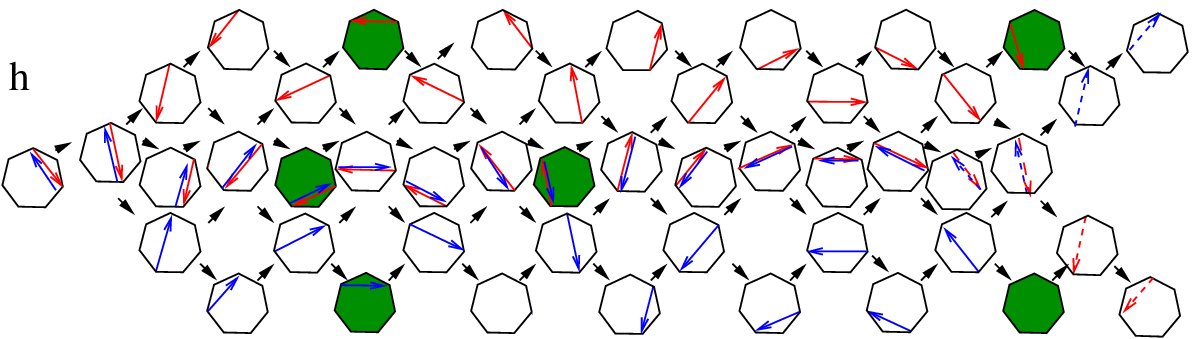}
\end{subfigure}\qquad \caption{$\sigma$-symmetric cluster configurations of $\Pi$.}
\label{g'}
\end{figure}

\begin{lemma}\label{l23}
Let $\mathcal{T}$ be a cluster configuration belonging to $\mathcal{F}_2$.
Then $\rho(\mathcal{T})$, and $\sigma_6(\mathcal{T})$ are also
cluster configurations.
\end{lemma}

\begin{proof}
Apply the map $\rho$, resp. $\sigma_6$ to the cluster configuration of 
Lemma \ref{tilt_fig}. Because of the
shape of the $\mathrm{Ext}$-hammocks one indeed produces
cluster configurations. 
\end{proof}
Notice that  the set $\rho(\mathcal{T})$ of a cluster 
configuration $\mathcal{T}$ is 
an elements of the $\tau$-orbit of $\mathcal{T}$, 
while $\sigma_i(\mathcal{T})$ does not belong to any
$\tau$-orbit of a cluster configuration of $\mathcal{F}_2$ (nor of $\mathcal{F}_1$).

\subsection{Classification of cluster tilting sets of $\mathcal{C}_{E_6}$ }\label{tilte6}

From \cite[Prop. 3.8]{FZ_y} and \cite[Thm. 4.5]{BMRRT} 
we know that there are 833 cluster tilting sets 
in $\mathcal{C}_{E_6}$.
In the next result we give a complete geometric
classification of all cluster tilting sets in $\mathcal{C}_{E_6}$ in terms of cluster
configurations of $\Pi$.

\begin{thm} \label{count} 
 In $\Pi$
\begin{itemize}
 \item 350 different cluster configurations have one long paired diagonal, and 
  they arise from $\mathcal{F}_1$ through $\tau$.
 \item  483 other cluster configurations arise from  
 $\mathcal{F}_2$ through  $\sigma$ and $\tau$.
\end{itemize}
\end{thm}
\begin{cor}\label{cor_count}
In $\Pi$:
\begin{itemize}
 \item 224  cluster configurations
have precisely one short paired diagonal,
 \item 175  cluster configurations have precisely two short paired diagonals,
 \item 84  cluster configurations have no paired diagonals of $\Pi$.
\end{itemize}
All these cluster configurations are different.
\end{cor}

\begin{proof}[Proofs of Theorem \ref{count} and Corollary \ref{cor_count}]

The first part of the claim follows from Lemma \ref{cor}. In fact, 
we saw that for each long paired diagonal $L_p$ 
there are $25$ ways to triangulate
one of the two pentagons $\Pi_5$ with single coloured diagonals. 
Moreover, there are two ways to triangulate
$\Pi_4$ with short paired diagonals. Thus, each $L_p$ gives rise to $50$ 
different cluster configurations.
Since there are $7$ choices for $L_p$ in $\Pi$, the first claim follows.

For the second part of the claim the idea is to consider different cases, 
depending on the number of short
paired diagonals leading to cluster configurations.
First case: the only paired diagonal of $\mathcal{T}$ is a short one. 
Then $\mathcal{T}$ arises from the collection 
highlighted in (a),(b) or (c) in Figure \ref{tilt}, 
up to $\tau$-shifts and the $\sigma$-symmetry of Lemma \ref{l23}. 
Moreover, after Lemma \ref{l1} for each coloured oriented 
diagonal $[i,i-3]_c$, $c\in\{R,B\}$ in $\mathcal{T}$ 
there two possible choices of neighbouring short single diagonals in $\mathcal{T}$. 
Consequently, up to $\tau$-shifts, the there are 4 different cluster configurations
arising from  a collection of diagonals as in (a). Similarly for (b). 
The collection in (c) gives rise to  
8 different cluster configurations up to $\tau$-shifts, as there are 4 choices for
short single diagonals, and further 4 arise by taking the $\sigma$-symmetric case.

Summing up, the cluster configurations in (a), (b), (c) give 
rise to 224 different cluster configurations.

Second case: $\mathcal{T}$ has exactly two short paired diagonals. 
Then one distinguishes further into (d),(e) and (f) which have 
single coloured oriented 
diagonals of the form $[i,i+3]_c$, $c\in\{R,B\}$.
While the cluster configurations in (g),(i),(h) and (l) only 
contain short single diagonals of the form
$[i,i+2]_c$, $c\in\{R,B\}$. 
Proceeding as before, taking into account 
the symmetry $\sigma$ of Lemma \ref{l23}
one obtains the
claimed number.

In the third case we count the cluster configurations arising
from $\tau$-shifts of the cluster configurations in (m) and (n) 
having no paired diagonals.
As before we deduce that there are 28 the cluster configurations
 arising from $\tau$-shifts of (m) 
and 56 arising from $\tau$-shifts of the 4 cluster configurations in (n).
Together this gives 84 cluster configurations without paired diagonals.
\end{proof}

\subsection{The geometry of cluster configurations}\label{subconf}
In Section \ref{longpaired} we saw that cluster
configurations
in $\mathcal{F}_1$ correspond to triangulations 
with coloured oriented diagonals of
two copies of $\Pi$.  
Cluster configurations of $\mathcal{F}_2$
are not as simple to describe. In Theorem \ref{confthm} below 
we can give a general statement concerning the 
geometry of cluster configurations in $\Pi$.

\begin{thm}\label{confthm}
All 833 cluster tilting sets of $\mathcal{C}_{E_6}$ can be expressed as
configurations of six non-crossing coloured oriented single and paired  diagonals 
inside two heptagons.
\end{thm}

\begin{proof}
First, given a cluster configurations in $\mathcal{F}_1\cup\mathcal{F}_2$
we divide the coloured oriented diagonals inside two heptagons by colour. 
Paired diagonals appear in both
heptagons. 

Then we observe that all cluster configurations in $\mathcal{F}_1$ 
and $\mathcal{F}_2$ are crossing free. 
Moreover, the symmetry $\sigma$ 
produces new configurations of diagonals which are again
crossing free. Taking $\tau$-shifts 
only rotates the entire configurations inside the two heptagons, occasionally 
switching colours and orientation according to the action of $\tau$ inside $\Pi$. Hence the 
crossing free property is preserved under $\tau$-shifts and the claim follows.
\end{proof}
The cluster configuration
$\mathcal{T}=\{ [5,3]_R, [5,2]_R ,[5,1]_P,[5,6]_P, [3,5]_B,[2,5]_B\}$
is expressed in two heptagons 
in the center of Figure \ref{Exchange graph}, the numbering of the vertices of
one heptagon is highlighted in the figure. Paired diagonals 
appear in both heptagons and have labels.

The converse statement of Theorem \ref{confthm} is not true, as configurations of  
non-crossing coloured diagonals different then cluster configurations of $\Pi$
are not cluster tilting sets of $\mathcal{C}_{E_6}$.

\subsection{Mutations of cluster tilting objects in $\mathcal{C}_{E_6}$}

In this section we will see that in many cases it is possible to deduce the 
mutation process in $\mathcal{C}_{E_6}$ 
from the mutations process inside $\mathcal{C}_{A_4}$.
Moreover, we can deduce the mutation process for 
the families in $\mathcal{F}_1$ and $\mathcal{F}_2$ 
and taking $\tau$-shifts extend it for the remaining cluster configurations.

In the next definitions we 
indicate by $\overline{D}$ 
the unoriented single
diagonal corresponding to
a coloured oriented diagonal $D$ of $\Pi$.
\begin{de} 
Let $D_P$ be a paired oriented diagonal and let
$\mathcal{T}$ be a cluster configuration containing $D_P$.
The paired diagonal $D_P^*$ is the flip-complement of $D_P$, if
$\overline{D}_P$ and $\overline{D^*}_P$ are related by a flip.
\end{de}
\begin{de}
Let $D_S$ be a single coloured oriented diagonal and let
$\mathcal{T}$ be a cluster configuration containing $D_S$.
The single coloured oriented diagonal $D_S^*$  
is the {\em flip-complement} of $D_S$ in $\mathcal{T}$, if $\overline{D}_S$ and $\overline{D^*}_S$ 
are related by a flip and $\mathcal{T}\backslash D_S \cup {D_S^*}$ is a cluster configuration.
\end{de}
Notice that if a flip-complement exists then the colour and orientation
is uniquely determined by Lemma \ref{cor} and Lemma \ref{l1}.

\begin{prop}\label{378} 
Let $D_L$ be a long coloured oriented diagonal of $\Pi$ giving 
rise to a cluster configuration. 
\begin{itemize}
 \item If $D_L$ is paired: single coloured diagonals triangulating $\Pi_5$, and 
 paired diagonals triangulating $\Pi_4$
 have a flip-complement. 
 \item If $D_L$ is single: single diagonals triangulating $\Pi_4$
 have a flip-complement.
 \end{itemize}
 \end{prop}

\begin{proof}
Let $D_L$ be a long coloured oriented 
diagonal dividing $\Pi$ into the quadrilateral $\Pi_4$ and the pentagon $\Pi_5$.
In Lemma \ref{cor} we saw that if $D_L$ is paired, triangulating
$\Pi_4$ with a paired diagonal, and two copies of $\Pi_5$ with single
oriented diagonals, always gives a cluster configuration.
Hence removing a single diagonal of a copy of $\Pi_5$
or a paired diagonal triangulating $\Pi_4$
can only be completed to a cluster configuration in two ways,
namely with diagonals being flip-complement of each other.
If $D_L$ is single, the claim follows from Lemma \ref{l1}. 
\end{proof}

Further instances of the mutation process in $\mathcal{C}_{E_6}$ 
can be described by flips of coloured oriented diagonals in $\Pi$, 
but not all mutations
allow a description of this type.
This is unsurprising, as for example 
not all mutations in the cluster algebra $\mathbb{C}[Gr_{3,7}]$
can be described through Pl\"ucker relations, see \cite{Scott}.

With Figure \ref{tilt} it is not hard to 
deduce all the remaining mutations occurring.
For example, the cluster configuration in (a) can be mutated to 
(b),(c),(e),(l),(m) and to the flip of 
a diagonal inside one light grey coloured heptagon.
Similarly, we have a list for the other families. 
Some instances of the more complicated geometric exchanges can be found 
on the upper pentagon of Figure \ref{Exchange graph}.

\subsection{Exchange graph}

In Figure \ref{Exchange graph} we display a part of the {\em exchange graph} 
of $\mathcal{C}_{E_6}$. For each heptagon appearing in the figure the numbering
of its vertices is as shown on the central heptagon. The vertices of the
graph correspond to cluster configurations, hence to cluster tilting sets
of $\mathcal{C}_{E_6}$, edges are 
drawn when two cluster configurations are related by a single mutation. 
In the two central heptagons of Figure \ref{Exchange graph}
the configuration of
$\mathcal{T}=\{ [5,3]_R, [5,2]_R ,[5,1]_P,[5,6]_P, [3,5]_B,[2,5]_B\}$ is displayed.
The $8$ neighbouring configurations 
are placed on the vertices of the two central pentagons 
sharing the vertex corresponding to $\mathcal{T}$.
These $8$ sets are obtained from $\mathcal{T}$
through repeated flips of single diagonals, as described in Proposition \ref{378}. 
The vertices of the left pentagon are obtained 
after mutating $[6,2]_P$ in $\mathcal{T}$. 

\subsection{Cluster tilted algebras}

Let $T=T_1\oplus\dots\oplus T_n$ be a cluster tilting object of $\mathcal{C}_{Q}$, then
 $\mathrm{End}_{\mathcal{C}_Q}(T)$ is the {\em cluster tilted algebra} of type $Q.$
The quiver $Q_T$ of $\mathrm{End}_{\mathcal{C}_Q}(T)$ has no loops nor 2-cycles and it 
encodes precisely the exchange matrix of the cluster associated to $T$, see \cite{BuMaRe} and \cite{CK}.
In $\mathcal{C}_{A_n}$ 
the quiver $Q_T$
can be read off from the triangulation $T$, see \cite{CCS}. The 
vertices of $Q_T$ are the diagonals of the triangulation and an arrow 
between $D_i$ and $D_j$ is drawn, 
whenever $D_i$ and $D_j$ bound a common triangle. 
The orientation of the arrow is $D_i\rightarrow D_j$, 
if $D_j$ is  linked to $D_i$
by an anticlockwise rotation around 
the common vertex. 

Extending the definition of $Q_T$ to the case at hand, the quivers 
corresponding to the cluster tilting sets of Figure \ref{Exchange graph} can be deduced. 
One could also read off the quivers, and relations, directly from $\Gamma$ 
by determining the spaces $\mathrm{Hom}_{\mathcal{C}_{E_6}}(T,T)$.

\begin{figure}
\includegraphics[scale=0.50, angle=90]{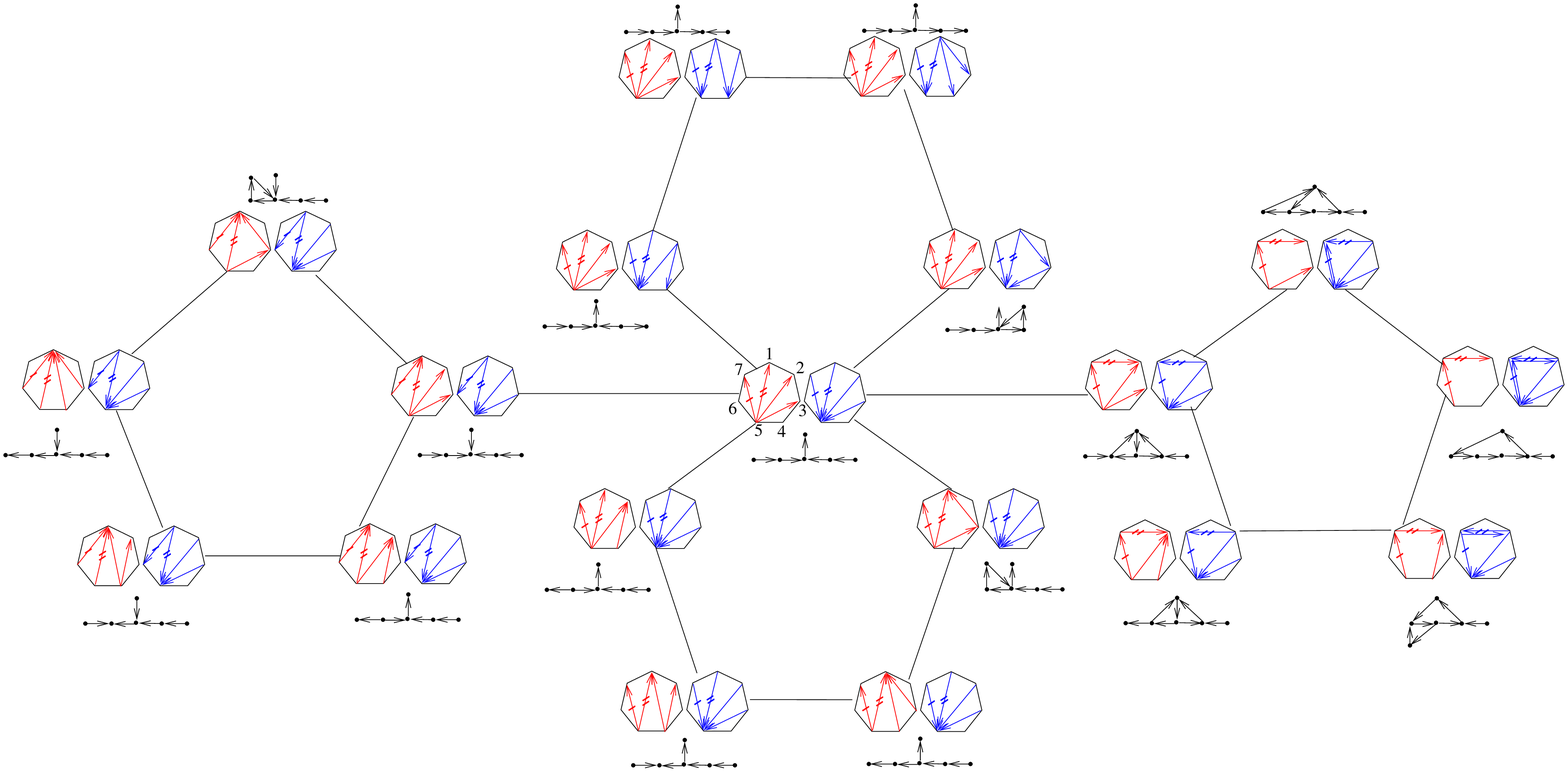}
\caption{ Part of the exchange graph for a cluster category of type $E_6$. 
At each vertex of the graph, the diagonals with the same labels are identified.}
\label{Exchange graph}
\end{figure}
\section{Applications and further directions}

\subsection{Symmetric cluster configurations and the cluster algebra of type ${F_4}$}

Cluster categories arising
from valued quivers have been first categorified algebraically in \cite{Demonet}.
The aim of this section is to use the geometric description of the cluster category
$\mathcal{C}_{E_6}$ to categorify geometrically the cluster
algebra of type $F_4$.

Consider again the map $\rho$ given by a simultaneous 
change of colour and orientation
of coloured oriented diagonals of $\Pi$. 
\begin{de}
A cluster configuration $\mathcal{T}$ in $\Pi$ is {\em $\rho$-symmetric}
if $\mathcal{T}=\rho(\mathcal{T})$. 
\end{de}
Since we know all cluster configurations of $\Pi$, see Theorem \ref{count}, we can deduce that there
are only three types of $\rho$-symmetric cluster configurations in $\Pi$. First, cluster configurations of $\Pi$
projecting to triangulations of $\Pi$ consisting of unoriented arcs through 
$\mathcal{C}_{E_6}\rightarrow \mathcal{C}_{A_4}$. These arise from $\mathcal{F}_1$ and
from $\tau$-shifts of the configurations in Figure \ref{tilt}(d).
Second, $\tau$-shifts of both, the cluster configuration in 
Figure \ref{tilt}(g) and the $\sigma$-symmetric configuration of Figure \ref{tilt}(g').
Third,  $\tau$-shifts of the configuration in Figure \ref{tilt}(h). 
We call the first $\rho$-symmetric cluster configurations of type \textbf{T},
the second of type \textbf{C} and the third of type \textbf{L}.
Moreover, we refer to the double short diagonals in a configuration of type \textbf{C}
as middle diagonals. 

In Figure \ref{confcl} configurations of type \textbf{L} and \textbf{C} are illustrated.
The blue diagonals have opposite orientation with respect to the ones shown in
the figure and are omitted. The
two $\rho$-symmetric cluster configurations of type \textbf{C} on the right side of the figure
are related through the action of $\sigma_6$, already defined in Section \ref{sigma}.
\begin{figure}
\includegraphics[scale=0.5]{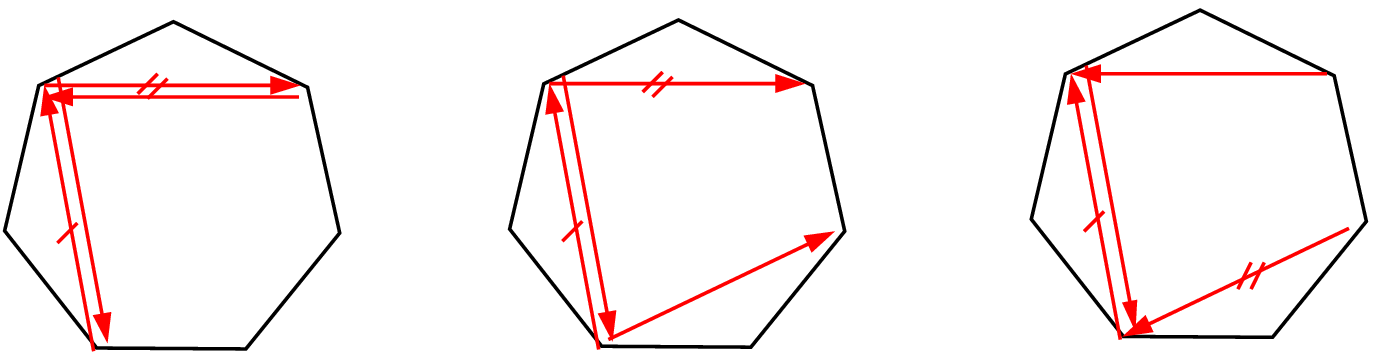}
\caption{$\rho$-symmetric cluster configurations of type \textbf{L} and \textbf{C}.}
\label{confcl}
\end{figure}
Notice that for $\rho$-symmetric cluster configurations
the quiver $Q_T$ of $\mathrm{End}_{\mathcal{C}_{E_6}}(T)$
is symmetric. 

In the next result we show
that mutations of cluster tilting objects in 
$\mathcal{C}_{E_6}$ preserves the $\rho$-symmetry of the cluster configuration.
 
Let $\mathcal{T}_{\rho}$  be a $\rho$-symmetric cluster configuration of 
$\Pi$ and denote by $D^*$ the unique complement in
$\mathcal{C}_{E_6}$ of $D$ in $\mathcal{T}_{\rho}$.
Then we observe that $\rho$-symmetric cluster configurations always have two 
$\rho$-orbits consisting of paired diagonals and two consist of
single diagonals of opposite colour and orientation, thus we can assume
$\mathcal{T}_{\rho}:=\{D_{P_1},D_{P_2},D_{S_1},\rho(D_{S_1}),D_{S_2},\rho(D_{S_2})\}$.
Keeping the notation as above, we have the following result.

\begin{prop}\label{mutrhosym} Let $1\leq k,l\leq 2$, then
$$\mathcal{T}_{\rho}\backslash D_{P_k}\cup D_{P_k}^* \textrm{ and } 
\mathcal{T}_{\rho}\backslash \{D_{S_l},\rho(D_{S_l})\} \cup \{D_{S_l}^*,\rho(D^*_{S_l})\}$$
are $\rho$-symmetric cluster configurations of $\Pi$.
\end{prop}
\begin{proof}
The claim follows from the mutation rule of $\mathcal{C}_{E_6}$ and the 
symmetry of the configurations. Moreover, since there are always two paired diagonals
and two $\rho$-orbits of single diagonals in $\mathcal{T}_{\rho}$ we deduce that
paired diagonals are exchanged with paired diagonals, and 
single diagonals with single diagonals.
\end{proof}
We call the $\rho$-symmetric cluster 
configuration obtained from $\mathcal{T}_{\rho}$ of Proposition \ref{mutrhosym}
the {\em mutation of $\mathcal{T}_{\rho}$ at
$D_{P_k}$, resp. at $\{D_{S_l},\rho(D_{S_l})\}$}, $1\leq k,l\leq 2$.
We refer to it by $\mathcal{T}_{\rho}^*$.
 
Let $\mathcal{A}_{F_4}$ be the cluster algebra
associated to a root system of type $F_4$. 
Then we are able to prove the claimed result.

\begin{prop}
There is a bijection 
$$
\{\rho\textrm{-symmetric cluster configurations in } \Pi\}\rightarrow \{\textrm{clusters in }\mathcal{A}_{F_4}\}
$$
compatible with mutations.
\end{prop}
\begin{proof}
Clearly there is a bijection between the $\rho$-orbits of 
coloured oriented diagonals in $\Gamma$ and the 24
cluster variables of  $\mathcal{A}_{F_4}$. 

Proceeding as in the proof of Theorem \ref{count}
we deduce that there are 105 $\rho$-symmetric cluster configurations in $\Pi$: 
84 are of type \textbf{T}, 14 of type \textbf{C} and 7 of type \textbf{L}.
From  \cite[Prop. 3.8]{FZ_y} we know that this number coincides with the 
number of clusters in the cluster algebra of type $F_4$, $\mathcal{A}_{F_4}$.

In addition,
the mutation rule in $\mathcal{C}_{E_6}$ induces a unique mutation rule
for $\rho$-symmetric cluster configurations of $\Pi$. This can simply be seen 
performing the mutation case by case. 
The mutation in $\mathcal{C}_{E_6}$ on $\rho$-symmetric cluster configurations
then agrees with the mutation of the cluster algebra  $\mathcal{A}_{F_4}$, 
since the exchange graph associated to $\mathcal{A}_{F_4}$ is regular
of degree 4, see  \cite[Thm 1.15]{FZ_y}. Thus there is essentially just one possible 
mutation.
\end{proof}
The 84 $\rho$-symmetric cluster configurations of type \textbf{T} in $\mathcal{C}_{E_6}$
are in 2:1 correspondence with the triangulations of $\Pi$, thus
with the cluster tilting sets of the cluster
category of type $\mathcal{C}_{A_4}$.

Geometrically the mutations of $\rho$-symmetric cluster 
configurations in $\Pi$ are described by the following
three moves. Let $c\in\{R,B\}$ and let $i$ a vertex of $\Pi$. 
We call a triangle bounded by coloured oriented
diagonals in $\Pi$ {\em internal} if all its edges are different then boundary edges.

\begin{description}
\item[L-C] In a configuration of type \textbf{L} the diagonals $\{[i+3,i+1]_c,\rho([i+3,i+1]_c)\}$  
exchange with $\{[i-1,i-3]_c,\rho([i-1,i-3]_c)\}$ in a configuration of type \textbf{C}. 
Similarly $[i+1,i+3]_P\leftrightarrow[i-3,i-1]_P$.\\
\item[C-T] In a configuration of type \textbf{C} the middle diagonals $\{[i+2,i]_c,\rho([i+2,i]_c)\}$  
exchange with $\{[i,i-3]_c,\rho([i,i-3]_c)\}$ bounding an internal triangle
in a configuration of type \textbf{T}. Similarly 
$[i-2,i]_P\leftrightarrow[i,i+3]_P$.\\
\item[T-T] All diagonals in a configuration of type \textbf{T} not 
bounding an internal triangle,
are exchanged with the usual flip rule. 
The orientation of the 
new diagonal is uniquely determined by the type of diagonal one exchanges. More precisely, the orientation
is such that paired diagonals are exchanges with paired diagonals, and single with single.  
\end{description}
\begin{figure}[H]
\centering
\psfragscanon
\psfrag{r}[][][0.75]{L-C}
\psfrag{b}[][][0.75]{C-T}
\begin{subfigure}[b]{0.25\textwidth}
\includegraphics[width=\textwidth]{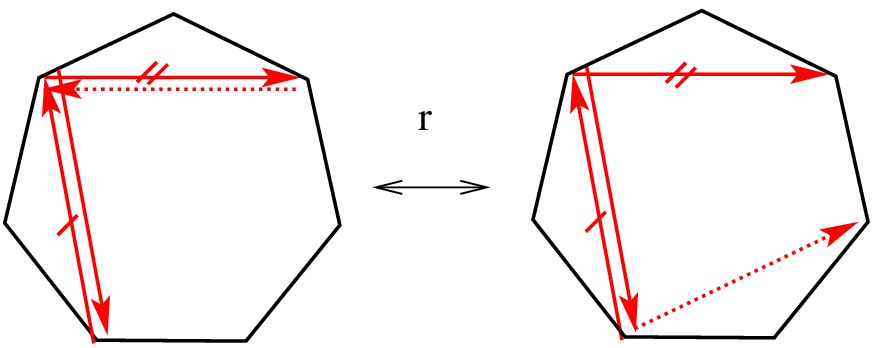}
\end{subfigure}\qquad\qquad 
\begin{subfigure}[b]{0.25\textwidth}
\includegraphics[width=\textwidth]{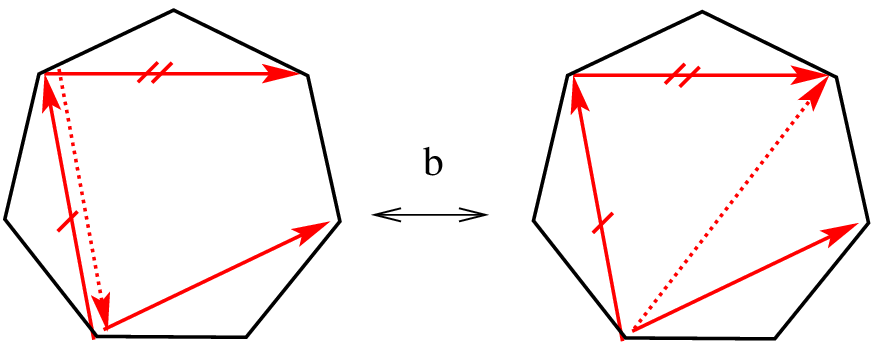}
\end{subfigure}
\caption{Mutations of $\rho$-symmetric cluster tilting configurations.}\label{L-C}
\end{figure}

In Figure \ref{L-C} we illustrate a mutation
between a $\rho$-symmetric cluster configuration of type \textbf{L}
and of type \textbf{C}, as well as a mutation between
a configuration of type \textbf{C} and type \textbf{T}.
The diagonals in dotted 
lines are complements of each other and 
as before, paired diagonals are labelled.

\subsection{Cluster tilting sets in $\mathcal{C}_{T_{r,s,t}}$}\label{end}

The construction of $\Gamma_{r,s,t}^+$ and $\Gamma_{r,s,t}^-$ of Section \ref{stabtrans}
was motivated by the following idea: glue
two copies of the AR-quiver of $\mathcal{C}^2_{A_{r+t+1}}$ 
along two disjoint $\tau$-orbits. This resulted in pairing diagonals (as in Section \ref{pairs}).
In this context one can ask.
\begin{problem}
How do categorical properties of the original category behave under this gluing operation?
How do cluster tilting sets behave under this operation?
\end{problem}

The  cluster categories $\mathcal{C}_{E_7}$ and $\mathcal{C}_{E_7}$  have finitely
many cluster tilting set and they can be described 
as cluster configurations of coloured oriented single and paired
diagonals inside a 10-gon, resp. a 16-gon, in a similar way as we did in Section \ref{tilte6}.
From these configurations we can identify again a family (denoted by $\mathcal{F}_1$ previously)
of cluster configurations giving rise to triangulations of regions homotopic to
a heptagon and a octagon, resp. a heptagon and a nonagon, compare with
Section \ref{longpaired} and Lemma \ref{cor}. 
In Figure \ref{Fig7} a cluster tilting set of $\mathcal{C}_{E_7}$  
arising from a projection of the projective modules
in $\mathrm{mod}k E_7$ is represented. Similarly, in Figure \ref{Fig8}
a cluster tilting set of $\mathcal{C}_{E_8}$ is represented.

\begin{figure}[H]
\centering
\psfragscanon
 \psfrag{1}[][][0.45]{$i$}
  \psfrag{2}[][][0.45]{$i+1$}
  \psfrag{3}[][][0.45]{$i+2$}
  \psfrag{4}[][][0.45]{$i+3$}
  \psfrag{5}[][][0.45]{$i+4$}
  \psfrag{6}[][][0.45]{$i+5$}
  \psfrag{7}[][][0.45]{$i+6$}
   \psfrag{8}[][][0.45]{$i+7$}
   \psfrag{9}[][][0.45]{$i+8$}
\begin{subfigure}[b]{0.25\textwidth}
\includegraphics[width=\textwidth]{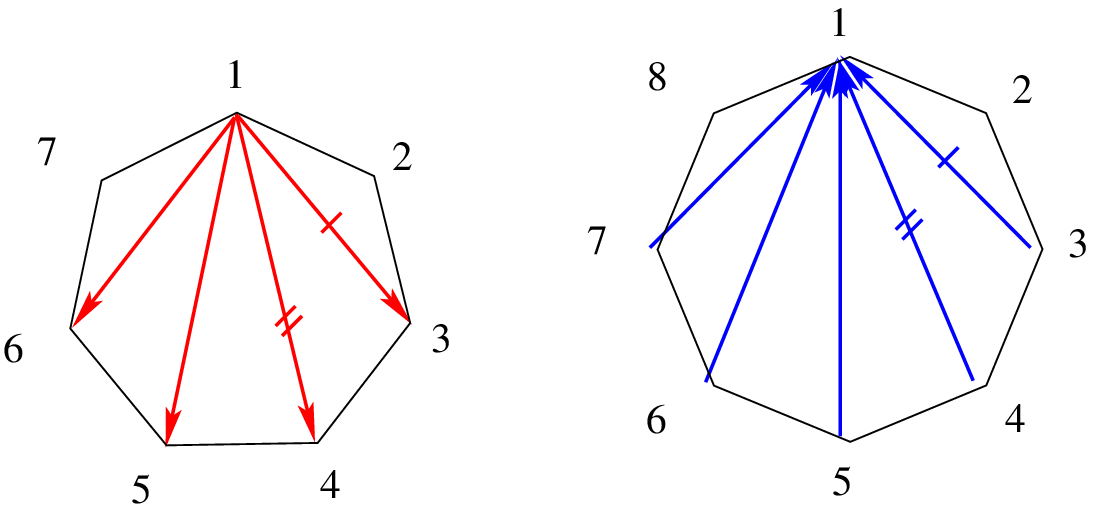}
\end{subfigure}
\caption{A cluster tilting set
of $\mathcal{C}_{E_7}$.}
\label{Fig7}
\begin{subfigure}[b]{0.25\textwidth}
\includegraphics[width=\textwidth]{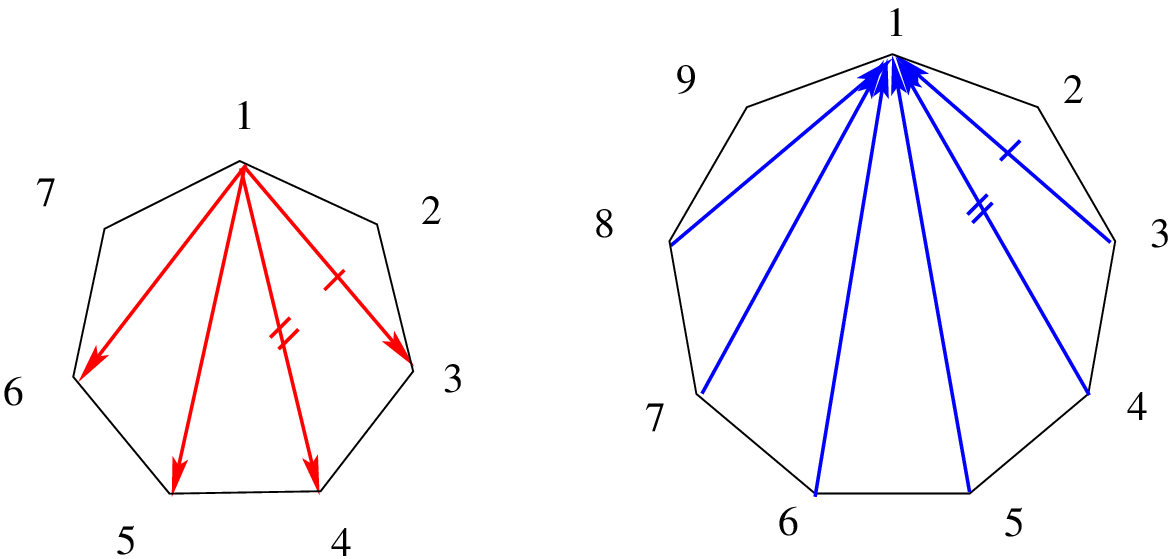}
\end{subfigure}
\caption{A cluster tilting set of $\mathcal{C}_{E_8}$.}
\label{Fig8}
\end{figure}

When $T_{r,s,t}$ is not of Dynkin type, $\mathcal{C}_{T_{r,s,t}}$  
has infinitely many cluster tilting sets. With our geometric approach, a number of 
cluster tilting sets of $\mathcal{C}_{T_{r,s,t}}$
can be expressed as 
configurations of coloured 
oriented single and paired 
diagonals inside the $(n+3)$-gon $\Pi$, where 
$n\geq\max\{r+t+1,r+s+1\}$.
Moreover, if one considers the projections
$\mathbb Z T_{r,s,t}\rightarrow\mathbb Z A_{r+t+1}$ and  
$\mathbb Z T_{r,s,t}\rightarrow\mathbb Z A_{s+t+1}$
one can describe the $\mathrm{Ext}$-hammocks in $\mathbb Z T_{r,s,t}$ using
the $\mathrm{Ext}$-hammocks in $\mathbb Z A_{r+t+1}$, resp. $\mathbb Z A_{s+t+1}$,
as we did in Subsection \ref{lift}. This enable us to express a number
of cluster configurations of $\mathcal{C}_{T_{r,s,t}}$
as triangulations of a $(s+t+4)$- and a $(r+t+4)$-gon. 

Finally, since in type $A$
all quivers obtained through Fomin-Zelevinsky quiver
mutations are known, we expect that our 
model can be used to understand the
quivers in the mutation class of an orientation of $T_{r,s,t}$, 
as well as 
parts of the exchange graph of 
$\mathcal{C}_{T_{r,s,t}}$. 

\subsection{Almost positive roots}
Consider the set of almost positive roots associated to
a root system $\Phi$. This set, denoted by $\Phi_{\geq -1}$,
consists of all positive roots together with all 
negative simple roots of $\Phi$. 

For an initial choice of a triangulation of a regular $(t+3)$-gon, 
an explicit bijection between all (unoriented) diagonals of 
the polygon and the 
set $\Phi_{\geq -1}$ of type $A_t$ was given in \cite{CCS}. 

From Theorem \ref{eqofcat},
together with \cite[Prop. 4.1]{BMRRT}
it follows that
there is a bijection between the vertices of $\Gamma_{r,s,t}^\pm$ and the 
set of almost positive roots $\Phi_{\geq -1}$ 
of the root system of type $E_6$, $E_7$ and $E_8$.

\begin{problem}
Describe geometrically the bijection between the coloured 
oriented single and paired diagonals in a
7-,10-, resp.16-gon and
the almost positive roots  of 
the root system of type $E_6$, $E_7$, resp. $E_8$.
\end{problem}


\textbf{Acknowledgments: } 
I would like to thank Karin Baur, Giovanni Felder and Robert Marsh for all the 
inspiring discussions we had and for
the many very helpful comments. I would also like to 
thank an anonymous referee for helpful suggestions. 

The author was partially supported by the Swiss National 
Science Foundation Grant Number PDFMP2127430.

\end{document}